\documentclass[12pt]{amsart}
\allowdisplaybreaks
\usepackage[english]{babel}
\usepackage[pdftex,textwidth=400pt,marginratio=1:1]{geometry}
\usepackage{amsfonts}
\usepackage[dvips]{graphics}
\usepackage[colorinlistoftodos]{todonotes}
\usepackage{amsmath}
\usepackage{amsthm}
\usepackage{amssymb}
\usepackage{bbm}
\usepackage{cancel}
\usepackage{color}
\usepackage[curve]{xypic}
\usepackage{graphicx}
\usepackage{hyperref}

\newtheorem{theorem}{Theorem}[section]

\newtheorem{proposition}[theorem]{Proposition}
\newtheorem{lemma}[theorem]{Lemma}
\newtheorem{conjecture}[theorem]{Conjecture}

\theoremstyle{definition}
\newtheorem{definition}[theorem]{Definition}

\newtheorem{remark}[theorem]{Remark}

\theoremstyle{property}

\newtheorem{Assumption}[theorem]{Assumption}

\DeclareFontFamily{OT1}{rsfs}{}
\DeclareFontShape{OT1}{rsfs}{n}{it}{<-> rsfs10}{}
\DeclareMathAlphabet{\curly}{OT1}{rsfs}{n}{it}
\newcommand{\dR}{\mathbf{R}}

\renewcommand\L{\mathcal L}

\renewcommand\O{\mathcal O}
\newcommand\PP{\mathbb P}

\newcommand\ccL{\mathcal L}

\newcommand\C{\mathbb C}

\newcommand\FF{\mathbb F}

\newcommand\Q{\mathbb Q}

\newcommand\R{\mathbb R}

\newcommand\Z{\mathbb Z}
\newcommand\cZ{\mathcal Z}

\newcommand\CM{\mathrm{CM}}

\newcommand\mov{\mathrm{mov}}

\newcommand\Exp{\mathrm{Exp}}

\newcommand\vir{\mathrm{vir}}

\newcommand\DT{\mathrm{DT}}
\newcommand\PT{\mathrm{PT}}

\newcommand\td{\mathrm{td}}
\newcommand\rk{\operatorname{rk}}

\newcommand\tr{\operatorname{tr}}

\newcommand\ch{\operatorname{ch}}

\newcommand\Hom{\operatorname{Hom}}
\renewcommand\hom{\mathcal{H}{\it{om}}}
\newcommand\Ext{\operatorname{Ext}}

\newcommand\Hilb{\operatorname{Hilb}}

\newcommand\mdot{{\scriptscriptstyle\bullet}}

\newcommand\INTO{\ar@{^{(}->}[r]}

\DeclareRobustCommand{\SkipTocEntry}[4]{}

\title[{$K$-theoretic DT/PT for toric Calabi-Yau 4-folds}]
{$K$-theoretic DT/PT correspondence for \\ toric Calabi-Yau 4-folds}
\date{}
\author{Yalong Cao}
\address{RIKEN Interdisciplinary Theoretical and Mathematical Sciences Program (iTHEMS), 2-1, Hirosawa, Wako-shi, Saitama, 351-0198, Japan}
\email{yalong.cao@riken.jp}
\author{Martijn Kool}
\address{Mathematical Institute, Utrecht University, P.O.~Box 80010 3508 TA Utrecht, The Netherlands}
\email{m.kool1@uu.nl}
\author{Sergej Monavari}
\address{Mathematical Institute, Utrecht University, P.O.~Box 80010 3508 TA Utrecht, The Netherlands}
\email{s.monavari@uu.nl}

\begin{document}

\maketitle
\begin{abstract}
Recently, Nekrasov discovered a new ``genus'' for Hilbert schemes of points on $\C^4$. 
We extend its definition to Hilbert schemes of curves and moduli spaces of stable pairs, and
conjecture a $K$-theoretic DT/PT correspondence for toric Calabi-Yau 4-folds. 
We develop a $K$-theoretic vertex formalism, which allows us to verify our conjecture in several cases.

Taking a certain limit of the equivariant parameters, we recover the cohomological DT/PT correspondence for toric Calabi-Yau 4-folds recently 
conjectured by the first two authors. 
Another limit gives a dimensional reduction to the $K$-theoretic DT/PT correspondence for toric 3-folds conjectured by Nekrasov-Okounkov. 

As an application of our techniques, we find a conjectural formula for the generating series of $K$-theoretic stable pair invariants of $\textrm{Tot}_{\mathbb{P}^1}(\mathcal{O}(-1) \oplus \mathcal{O}(-1) \oplus \O)$. Upon dimensional reduction to the resolved conifold, we recover a formula which was recently proved by Kononov-Okounkov-Osinenko.
\end{abstract}

\vspace{-0.1cm}

\tableofcontents

\section{Introduction}

Two recent developments are Donaldson-Thomas type invariants of Calabi-Yau 4-folds (e.g.~\cite{BJ, CL1, CL2, CGJ, CK1, CK2, CKM, CMT1, CMT2, CT1, CT2, CT3, CT4, Nek, NP, OT}) and $K$-theoretic virtual invariants introduced by Nekrasov-Okounkov (e.g.~\cite{NO, O, Nek, Afg, Arb, FMR, Tho}). 
Let $X$ be a complex smooth quasi-projective variety, $\beta \in H_2(X,\Z)$, and $n \in \Z$. We consider the following moduli spaces:
\begin{itemize}
\item $I:=I_n(X,\beta)$ denotes the Hilbert scheme of proper closed subschemes $Z \subseteq X$ of dimension $\leqslant 1$ satisfying $[Z] = \beta$ and $\chi(\O_Z)=n$,
\item $P:=P_n(X,\beta)$ is the moduli space of stable pairs $(F,s)$ on $X$, where $F$ is a pure 1-dimensional sheaf on $X$ with proper scheme theoretic support in class $\beta$, $\chi(F) = n$, and $s \in H^0(F)$ has 0-dimensional cokernel. 
\end{itemize} 

For proper Calabi-Yau 3-folds, both spaces have a symmetric perfect obstruction theory. The degrees of the virtual classes are known as (rank one) 
Donaldson-Thomas and Pandharipande-Thomas invariants. Their generating series are related by the famous DT/PT correspondence conjectured by Pandharipande-Thomas \cite{PT1} and proved by Bridgeland \cite{Bri} and Toda \cite{Tod}.

For proper Calabi-Yau 4-folds, $I$ and $P$ still have an obstruction theory. Denote the universal objects by $\cZ \subseteq I \times X$ and $\mathbb{I}^{\mdot} = \{\O_{P \times X} \rightarrow \FF\}$. Then
\begin{align*}
T_I^{\vir} = \dR \hom_{\pi_I}(I_{\mathcal{Z}}, I_{\mathcal{Z}})_0[1], \quad T_P^{\vir} = \dR\hom_{\pi_P}(\mathbb{I}^{\mdot}, \mathbb{I}^{\mdot})_0[1],
\end{align*}
where $(\cdot)_0$ denotes trace-free part and $\dR \hom_\pi := \dR \pi_* \circ \dR\hom$. These obstruction theories are \emph{not perfect}, so the machineries of Behrend-Fantechi \cite{BF} and Li-Tian \cite{LT} do not produce virtual classes on the moduli spaces. Nonetheless, there exist virtual classes $[I]^{\vir}_{o(\ccL)} \in H_{2n}(I,\mathbb{Z})$, $[P]^{\vir}_{o(\ccL)} \in H_{2n}(P,\mathbb{Z})$ (ref. \cite{CMT2, CK2}) in the sense of Borisov-Joyce \cite{BJ}, which involves derived algebraic geometry \cite{PTVV} and derived differential geometry. These virtual classes depend on the choice of an orientation $o(\ccL)$, i.e.~the choice of a square root of the isomorphism
$$
Q : \ccL \otimes \ccL \rightarrow \O, \quad \mathcal{L} := \det \dR \hom_{\pi}(\mathbb{E},\mathbb{E})
$$
induced by the Serre duality pairing (here $\mathbb{E}=I_{\cZ}$ or $\mathbb{I}^{\mdot}$ respectively).

\subsection{Nekrasov genus} 

In this paper, $X$ is a toric Calabi-Yau 4-fold\,\footnote{I.e.~a smooth quasi-projective toric 4-fold $X$ satisfying $K_X \cong \O_X$, $H^{>0}(\O_X) = 0$, and such that every cone of its fan is contained in a 4-dimensional cone.}. 
Since $X$ is non-proper, the moduli spaces $I, P$ are \emph{in general} non-proper and we define invariants by a localization formula. There are interesting cases for which $P$ is proper, e.g.~when 
$X = \mathrm{Tot}_{\PP^2}(\O(-1) \oplus \O(-2)), \mathrm{Tot}_{\PP^1 \times \PP^1}(\O(-1,-1) \oplus \O(-1,-1))$. Denote by $(\C^*)^4$ the dense open torus of $X$ and let $T \subseteq (\C^*)^4$ be the 3-dimensional subtorus preserving the Calabi-Yau volume form. Then 
$$
 I^T =  I^{(\C^*)^4},
$$
which consists of finitely many isolated reduced points \cite[Lem.~2.2]{CK2}. Roughly speaking, these are described by solid partitions (4D piles of boxes) corresponding to monomial ideals in each toric chart $U_\alpha \cong \C^4$ with infinite ``legs'' along the coordinate axes, which agree on overlaps $U_\alpha \cap U_\beta$. In general, the fixed locus $P^{(\C^*)^4}$ may not be isolated \cite{CK2}. Throughout this paper, whenever we consider a moduli space $P$ of stable pairs, we assume:

\begin{Assumption}\label{assumption}
$X$ is a toric Calabi-Yau 4-fold and $\beta \in H_2(X,\Z)$ such that $\bigcup_n P_n(X,\beta)^{(\C^*)^4}$ is at most 0-dimensional. 
\end{Assumption}
 
 When Assumption \ref{assumption} is satisfied, $P_n(X,\beta)^{T} = P_n(X,\beta)^{(\C^*)^4}$ for all $n$ and it consists of finitely many reduced points, which are combinatorially described in \cite[Sect.~2.2]{CK2}. 
 
 This assumption is equivalent to saying that in each toric chart $U_\alpha$ at most two infinite legs come together (Lemma \ref{fewlegs}). This is the case when $X$ is a local curve or local surface. If in each toric chart $U_\alpha$ at most three legs come together, $P^{(\C^*)^4}$ is isomorphic to a disjoint union of products of $\PP^1$'s (essentially by \cite{PT2}). This is the case when $X$ is a local threefold. In full generality, four infinite legs can come together in each toric chart $U_\alpha$. Then $P^{(\C^*)^4}$ is considerably more complicated; its connected components are cut out by incidence conditions from ambient spaces of the form $\mathrm{Gr}(1,2)^\ell \times \mathrm{Gr}(1,3)^m \times \mathrm{Gr}(2,3)^n$. In order to avoid moduli, we focus on the isolated case, though we expect the results of this paper can be generalized to the general setting.

At any fixed point $x=Z \in I^T$ or $x=[(F,s)] \in P^T$, $T$-equivariant Serre duality implies that the $T$-equivariant $K$-theory classes
$$
T_I^{\vir}|_x, \quad T_P^{\vir}|_x \in K_0^T(\mathrm{pt}) = \Z[t^{\pm 1}_1,t^{\pm 1}_2,t^{\pm 1}_3,t^{\pm 1}_4] / (t_1t_2t_3t_4-1)
$$
have square roots. Namely~there exist
$
\sqrt{T_I^{\vir}|_x}, \, \sqrt{T_P^{\vir}|_x} \in K_0^T(\mathrm{pt})
$
such that
\begin{align*}
T_I^{\vir}|_x = \sqrt{T_I^{\vir}|_x} + \overline{\sqrt{T_I^{\vir}|_x}}, 
\end{align*}
and similarly for $P$, where $\overline{(\cdot)}$ denotes the involution on $K_0^T(\mathrm{pt})$ induced by $\Z$-linearly extending the map $t_1^{w_1}t_2^{w_2}t_3^{w_3}t_4^{w_4} \mapsto t_1^{-w_1}t_2^{-w_2}t_3^{-w_3}t_4^{-w_4}$. These square roots are \emph{non-unique}. In what follows, we use the following notation
\begin{align*}
\Omega^{\vir}_I|_x :=  (T_I^{\vir}|_x)^{\vee}, \quad \quad K^{\vir}_I|_x :=  \det \Omega_I^{\vir}|_x, \quad \mathrm{Ob}_I|_x := h^1(T_I^{\vir}|_x).
\end{align*}
We denote the $T$-moving and $T$-fixed parts by
$$
N^{\vir}|_{x}= (T_I^{\vir}|_{x})^{\mov}, \quad (T_I^{\vir}|_{x})^{f}.
$$
We use similar notations for the stable pairs case. A choice of a square root of $T_I^{\vir}|_x$, $T_P^{\vir}|_x$ induces a square root $\sqrt{E}$ for each of the above complexes $E$.  

For any $T$-equivariant line bundle $L$ on $X$, we define 
\begin{equation} \label{taut}
L^{[n]} := \textbf{R}\pi_{I*} (\pi_X^* L \otimes \O_{\mathcal{Z}}), \quad \textbf{R}\pi_{P*} (\pi_X^* L \otimes \FF)
\end{equation}
on the moduli spaces $I$ and $P$. 
Here $\pi_X$, $\pi_I$ (resp.~$\pi_P$) are projections from $X\times I$ (resp.~$X\times P$) to the corresponding factor.
Moreover, for a $T$-equivariant vector bundle $E$ on any scheme $M$ with $T$-action, we define 
$$
\Lambda_\tau E = \sum_{i \geqslant 0} [\Lambda^i E] \, \tau^i \in K^0_T(M)[\tau], \quad \mathrm{Sym}_\tau E = \sum_{i \geqslant 0} [\mathrm{Sym}^i E] \, \tau^i  \in K^0_T(M)[[\tau]],
$$
where $K^0_T(M)$ denotes the $K$-group of $T$-equivariant locally free sheaves on $M$. This can be extended to $K^0_T(M)$ by $\Lambda_\tau([E]-[F]) := \Lambda_\tau(E) \cdot \mathrm{Sym}_{-\tau}(F)$ \cite[Sect.~4]{FG}.
Following Nekrasov \cite{Nek}, which deals with the case $I_n(\C^4,0)$, we define:
\begin{definition} \label{Nekgen}
We define the following ``Nekrasov genus'' of the moduli space $I := I_{n}(X,\beta)$. Consider an extra trivial $\C^*$-action on $X$ and let $\O \otimes y$ be the trivial line bundle with non-trivial $\C^*$-equivariant 
structure corresponding to primitive character $y$.
For any $T$-equivariant line bundle $L$ on $X$, we define
{\footnotesize{\begin{align*}
&I_{n,\beta}(L, y) := \chi\Big(I, \widehat{\O}^{\vir}_I \otimes \frac{\Lambda^{\mdot} (L^{[n]} \otimes y^{-1})} {(\det(     L^{[n]} \otimes y^{-1} ))^{\frac{1}{2}}  } \Big) 
:= \chi \Big(I^{T}, \frac{ \O^{\vir}_{I^T} \otimes \sqrt{K_I^{\vir}}^{\frac{1}{2}}|_{I^T} }{\Lambda^{\mdot} \sqrt{N^{\vir}}^{\vee}} \otimes \frac{\Lambda^{\mdot} (L^{[n]} \otimes y^{-1})} {(\det(     L^{[n]} \otimes y^{-1} ))^{\frac{1}{2}}  }  \Big) \\
&:=\sum_{Z \in I^{T}} (-1)^{o(\L)|_Z} e\left(\sqrt{\mathrm{Ob}_I|_Z}^f\right)  \frac{\ch\left(\sqrt{K_I^{\vir}|_{Z}}^{\frac{1}{2}}\right)}{\ch\left(\Lambda^{\mdot} \sqrt{N^{\vir}|_{Z}}^{\vee}\right)}   \frac{\ch(\Lambda^{\mdot} (L^{[n]}|_{Z} \otimes y^{-1}))}{\ch((\det(L^{[n]} |_{Z} \otimes y^{-1}))^{\frac{1}{2}})} \td\left(\sqrt{T_I^{\vir}|_{Z}}^{f}\right),
\end{align*}}}
${}$ \\
where $\Lambda^\mdot(\cdot) = \Lambda_{-1}(\cdot)$. Here the first two lines are suggestive notations and the third line is the actual definition. This definition depends on the choice of a sign $(-1)^{o(\ccL)|_Z}$ for each $Z \in I^T$. We suppress this dependence from the notation. All Chern characters $\ch(\cdot)$, the Euler class $e(\cdot)$, and Todd class $\td(\cdot)$ in this formula are $T \times \C^*$-equivariant ($T$ is the Calabi-Yau torus and $\C^*$ the trivial torus) and the invariant takes value in 
$\frac{\Q(t_1^{\frac{1}{2}},t_2^{\frac{1}{2}},t_3^{\frac{1}{2}},t_4^{\frac{1}{2}},y^{\frac{1}{2}})}{(t_1t_2t_3t_4-1)}$.
Different choices of square root of $T_I^{\vir}|_Z$ only change the contribution of $Z$ to the invariant by a sign, so this gets absorbed in the choice of sign $(-1)^{o(\ccL)|_Z}$.\footnote{When developing the vertex formalism in Section \ref{sec:taking roots}, we make an explicit choice of square root for each $T_I^{\vir}|_Z$ and $T_P^{\vir}|_Z$.}
We define $P_{n,\beta}(L, y)$ analogously replacing $I$ by $P$ and imposing Assumption \ref{assumption}. 
\end{definition}

\begin{remark}
When the the first version of this paper became public, the virtual structure sheaf and $K$-theoretic localization formula were not established yet in the setting of Calabi-Yau 4-folds. As a consequence, we \emph{defined} our invariants $I_{n,\beta}(L, y)$ (and $P_{n,\beta}(L, y)$) by the (expected) virtual localization formula as in Definition \ref{Nekgen}. Recently, the virtual structure sheaf and $K$-theoretic localization formula have been established by J.~Oh and R.~P.~Thomas \cite{OT} thereby vindicating the calculations in this paper. 
%When this paper is written, $I_{n,\beta}(L, y)$ (and $P_{n,\beta}(L, y)$) cannot be defined \emph{globally} as a holomorphic Euler characteristic on $I$, because the twisted virtual structure sheaf $\widehat{\O}^{\vir}_I$ has not been defined yet in the 4-fold setting (the first equality). So the paper is rather about the vertex formalism and its applications from working directly on the fixed locus. As is clear from our definition on the level of fixed points, we expect the twisted virtual structure sheaf to be of the form $\widehat{\O}^{\vir}_I = \O^{\vir}_I \otimes \sqrt{K_I^{\vir}}^{\frac{1}{2}}$ and it should have the property that a $K$-theoretic virtual localization formula  as in the first line of Definition \ref{Nekgen} holds. Recently, Oh-Thomas \cite{OT} is able to define the twisted virtual structure sheaf globally and prove the expected $K$-theoretic virtual localization formula.
\end{remark}

\begin{remark}
For any $Z \in I^T := I_n(X,\beta)^T$, $T_I^{\vir}|_Z$ does not contain any $T$-fixed terms with positive coefficient (Lemma \ref{TfixlocusDT}). Therefore
\begin{align*}
I_{n,\beta}(L, y) =\sum_{Z \in I^{T}} (-1)^{o(\mathcal{L})|_Z} \frac{\ch\left(\sqrt{K_I^{\vir}|_{Z}}^{\frac{1}{2}}\right)}{\ch\left(\Lambda^{\mdot} \sqrt{T_I^{\vir}|_{Z}}^{\vee}\right)}   \frac{\ch(\Lambda^{\mdot} (L^{[n]}|_{Z} \otimes y^{-1}))}{\ch((\det(L^{[n]} |_{Z} \otimes y^{-1}))^{\frac{1}{2}})}.
\end{align*}
When $T_I^{\vir}|_Z$ does not contain $T$-fixed terms with negative coefficient, this equality is clear since $T_I^{\vir}|_Z = N^{\vir}|_Z$ and $(T_I^{\vir}|_Z)^f=0$. When $T_I^{\vir}|_Z$ contains a $T$-fixed term with negative coefficient, both LHS and RHS are zero (since $e\big(\sqrt{\mathrm{Ob}_I |_Z}^f\big) = 0$). A similar statement holds in the stable pairs case (where we require Assumption \ref{assumption} and use Lemma \ref{PTfixedlocus}).
\end{remark}

\subsection{$K$-theoretic DT/PT correspondence}  

We show in Section \ref{vertex} that the invariants $I_{n,\beta}(L, y)$ and $P_{n,\beta}(L, y)$ can be calculated by a $K$-theoretic vertex formalism. The case $I_{n,0}(L, y)$ was originally established by Nekrasov \cite{Nek} and Nekrasov-Piazzalunga \cite{NP}, who also deal with the higher rank case. 
The case $I_{n,\beta}(L, y)$ is recently independently established in \cite{NP2}, who also deal with ideal sheaves of surfaces and the higher rank case (see Subsection \ref{otherwork} below). 
Our focus is on the $K$-theoretic DT/PT correspondence for toric Calabi-Yau 4-folds. In Section \ref{vertex}, we define the $K$-theoretic DT/PT 4-fold vertex\footnote{A priori the powers of $t_1,t_2,t_3,t_4$ in $\mathsf{V}_{\lambda\mu\nu\rho}^{\DT}(t,y,q)$, $\mathsf{V}_{\lambda\mu\nu\rho}^{\PT}(t,y,q)$ are half-integers. We prove in Proposition \ref{integerpowers} that they are always integers.} 
$$
\mathsf{V}_{\lambda\mu\nu\rho}^{\DT}(t,y,q) , \quad \mathsf{V}_{\lambda\mu\nu\rho}^{\PT}(t,y,q) \in \frac{\Q(t_1,t_2,t_3,t_4,y^{\frac{1}{2}})}{(t_1t_2t_3t_4-1)}(\!(q)\!),
$$
for any finite plane partitions (3D partitions) $\lambda, \mu, \nu, \rho$. In the stable pairs case, we require that at most two of $\lambda, \mu, \nu, \rho$ are non-empty (which follows from Assumption \ref{assumption} by Lemma \ref{fewlegs}). Roughly speaking, these are the generating series of $I_{n,\beta}(L, y)$, $P_{n,\beta}(L, y)$ in the case $X = \C^4$, $L = \O_{\C^4}$, and the underlying Cohen-Macaulay support curve is fixed and described by finite asymptotic plane partitions $\lambda, \mu, \nu, \rho$ (see Definition \ref{vertexdef}). The series $\mathsf{V}_{\lambda\mu\nu\rho}^{\DT}$, $\mathsf{V}_{\lambda\mu\nu\rho}^{\PT}$ depend on the choice of a sign at each $T$-fixed point. 

Before we phrase our DT/PT vertex correspondence, we discuss a beautiful conjecture by Nekrasov for $\mathsf{V}_{\varnothing\varnothing\varnothing\varnothing}^{\DT}$ \cite{Nek, NP}. We recall the definition of the plethystic exponential. For any formal power series $f(p_1, \ldots, p_r; q_1, \ldots, q_s)$ in $\Q(p_1, \ldots, p_r)[\![q_1, \ldots, q_s]\!]$,
its plethystic exponential is defined by
\begin{align*} 
\Exp(f(p_1, \ldots, p_r;q_1, \ldots, q_s)) &:= \exp\Big( \sum_{n=1}^{\infty} \frac{1}{n} f(p_1^n, \ldots, p_r^n;q_1^n, \ldots, q_s^n) \Big)
\end{align*}
viewed as an element of $\Q(p_1, \ldots, p_r)[\![q_1, \ldots, q_s]\!]$. Following Nekrasov \cite{Nek}, for any formal variable $x$, we define
\begin{equation*} 
[x] := x^{\frac{1}{2}} - x^{-\frac{1}{2}}.
\end{equation*}
\begin{conjecture}[Nekrasov] \label{Nekconj}
There exist unique choices of signs such that
\begin{align*}
\mathsf{V}_{\varnothing\varnothing\varnothing\varnothing}^{\DT}(t,y,q) = \Exp(\mathcal{F}(t,y;q)), \quad \mathcal{F}(t,y;q):= \frac{[t_1t_2][t_1t_3][t_2t_3][y]}{[t_1][t_2][t_3][t_4][y^{\frac{1}{2}}q][y^{\frac{1}{2}}q^{-1}]}, 
\end{align*}
where $\mathcal{F}(t,y;q) \in \frac{\Q(t_1,t_2,t_3,t_4,y^{\frac{1}{2}},q)}{(t_1t_2t_3t_4-1)}$ is expanded as a formal power series in $q$.
\end{conjecture}
See \cite{Nek} for the existence part. Here we conjecture the uniqueness part. 
We propose the following $K$-theoretic DT/PT 4-fold vertex correspondence:
\begin{conjecture}\label{K-conj intro}
For any finite plane partitions $\lambda, \mu, \nu, \rho$, at most two of which are non-empty, there are choices of signs such that 
$$
\mathsf{V}_{\lambda\mu\nu\rho}^{\DT}(t,y,q) = \mathsf{V}_{\lambda\mu\nu\rho}^{\PT}(t,y,q) \, \mathsf{V}_{\varnothing\varnothing\varnothing\varnothing}^{\DT}(t,y,q).
$$
Suppose we choose the signs for $\mathsf{V}_{\varnothing\varnothing\varnothing\varnothing}^{\DT}(t,y,q)$ equal to the unique signs in Nekrasov's conjecture \ref{Nekconj}. Then, at each order in $q$, the choice of signs for which LHS and RHS agree is unique up to an overall sign.
\end{conjecture}

We verify this conjecture in various cases for which $|\lambda| + |\mu| + |\nu| + |\rho| \leqslant 4$ and the number of embedded boxes is $\leqslant 3$ (for the precise statement, see Proposition \ref{verif}. This conjecture and the vertex formalism imply the following:
\begin{theorem} \label{globaltoricKDTPT}
Assume Conjecture \ref{K-conj intro} holds. Let $X$ be a toric Calabi-Yau 4-fold and $\beta \in H_2(X,\Z)$ such that $\bigcup_n P_n(X,\beta)^{(\C^*)^4}$ is at most 0-dimensional. Let $L$ be a $T$-equivariant line bundle on $X$. Then there exist choices of signs such that
$$\frac{\sum_{n} I_{n,\beta}(L, y)\, q^n}{\sum_{n} I_{n,0}(L, y)\, q^n}=\sum_{n} P_{n,\beta}(L, y)\, q^n.$$
\end{theorem}

One may wonder whether there are other $K$-theoretic insertions for which a DT/PT correspondence similar to Conjecture \ref{K-conj intro} holds. The most natural candidates are virtual holomorphic Euler characteristics $\chi(I, \widehat{\O}^{\vir}_I)$, $\chi(P, \widehat{\O}^{\vir}_P)$, or replacing $L$ in Definition \ref{Nekgen} by a higher rank vector bundle (or even $K$-theory classes of negative rank). However, we have \emph{not} found any other $K$-theoretic insertions that work and we believe that the insertion of Definition \ref{Nekgen} is special (see Remark \ref{rmk on other insertions} for the precise statement).

In Remark \ref{sign expec}, we present expected closed formulae for the unique signs (up to overall sign) of Conjecture \ref{K-conj intro}, which work for all the verifications done in this paper. This generalizes the sign formula obtained by Nekrasov-Piazzalunga, from physics methods, for Hilbert schemes of points on $\C^4$ \cite[(2.60)]{NP}.

We now discuss three limits, which were treated in the case of $I_n(\C^4,0)$ in \cite[Sect.~5.1--5.2]{Nek} (though we do not need the ``perturbative term'' of loc.~cit.).

\subsection{Dimensional reduction to 3-folds} 

Let $D$ be a smooth toric 3-fold\,\footnote{More precisely, a smooth quasi-projective toric 3-fold such that every cone of its fan is contained in a 3-dimensional cone.} and let $\beta \in H_2(D,\Z)$. Consider the following generating functions
$$
\sum_{n} \chi(I_n(D,\beta), \widehat{\O}^{\vir}_I) \, q^n, \quad \sum_{n} \chi(P_n(D,\beta), \widehat{\O}^{\vir}_P) \, q^n,
$$
where $\widehat{\O}^{\vir}_I = \O^{\vir}_I \otimes (K_I^{\vir})^{\frac{1}{2}}$, $\widehat{\O}^{\vir}_P = \O^{\vir}_P \otimes (K_P^{\vir})^{\frac{1}{2}}$ are the twisted virtual structure sheaves of $I_n(D,\beta)$, $P_n(D,\beta)$ introduced in \cite{NO}\,\footnote{In the 3-fold case, the invariants $\chi(I_n(D,\beta), \widehat{\O}^{\vir}_I)$, $\chi(P_n(D,\beta), \widehat{\O}^{\vir}_P)$ do not depend on the choice of square root $(K_I^{\vir})^{\frac{1}{2}}$, $(K_P^{\vir})^{\frac{1}{2}}$. This is because different choices of square roots have the same first Chern class (modulo torsion). See also \cite[Section 2.5]{Arb}.}.

The calculation of the $K$-theoretic DT/PT invariants of toric 3-folds is governed by the $K$-theoretic 3-fold DT/PT vertex \cite{NO, O, Arb}
$$
\mathsf{V}_{\lambda\mu\nu}^{\textrm{3D},\DT}(t,q), \quad \mathsf{V}_{\lambda\mu\nu}^{\textrm{3D},\PT}(t,q) \in \Q(t_1,t_2,t_3,(t_1t_2t_3)^{\frac{1}{2}})(\!(q)\!),
$$
where $\lambda, \mu, \nu$ are line partitions (2D partitions) determining the underlying $(\C^*)^3$-fixed Cohen-Macaulay curve and $t_1,t_2,t_3$ are the characters of the standard torus action on $\C^3$.

In the next theorem, $\lambda,\mu,\nu$ are line partitions in the $(x_2,x_3)$, $(x_1,x_3)$, $(x_1,x_2)$-planes respectively. Then $\lambda, \mu, \nu$ can be seen as plane partitions in $(x_2,x_3,x_4)$, $(x_1,x_3,x_4)$, $(x_1,x_2,x_4)$-space, respectively, by inclusion $\{x_4=0\} \subseteq \C^3$. 

Any plane partition $\lambda, \mu, \nu, \rho$ determine a $(\C^*)^4$-fixed Cohen-Macaulay curve on $\C^4$ with asymptotic profiles $\lambda, \mu, \nu, \rho$. The ideal sheaf of such a curve corresponds to a monomial ideal, which is described by a solid partition denoted by $\pi_{\CM}(\lambda,\mu,\nu,\rho)$ (this is explained in detail in Section \ref{fixlocus}). The renormalized volume of this solid partition is denoted by $|\pi_{\CM}(\lambda,\mu,\nu,\rho)|$ (Definition \ref{solid}).

\begin{theorem} \label{dimred intro}
Let $\lambda,\mu,\nu$ be any line partitions in the $(x_2,x_3)$, $(x_1,x_3)$, $(x_1,x_2)$-planes respectively. For any $T$-fixed subscheme $Z \subseteq \C^4$ with underlying maximal Cohen-Macaulay curve $C$ determined by $\lambda,\mu,\nu, \varnothing$, we choose its sign in Definition \ref{Nekgen} equal to $(-1)^{|\pi_{\CM}(\lambda,\mu,\nu,\varnothing)| + \chi(I_C / I_Z)}$, where $\chi(I_C / I_Z)$ equals the number of embedded points of $Z$.
For any $T$-fixed stable pair $(F,s)$ on $\C^4$ with underlying Cohen-Macaulay curve determined by $\lambda,\mu,\varnothing, \varnothing$, we choose its sign in Definition \ref{Nekgen} equal to $(-1)^{|\pi_{\CM}(\lambda,\mu,\varnothing,\varnothing)| + \chi(Q)}$, where $\chi(Q)$ denotes the length of the cokernel of $s$.
Then
\begin{align}
\begin{split} \label{dimredeqns}
\mathsf{V}_{\lambda\mu\nu\varnothing}^{\DT}(t,y,q)|_{y=t_4} = \mathsf{V}_{\lambda\mu\nu}^{\mathrm{3D},\DT}(t,-q), \quad \mathsf{V}_{\lambda\mu\varnothing\varnothing}^{\PT}(t,y,q)|_{y=t_4} = \mathsf{V}_{\lambda\mu\varnothing}^{\mathrm{3D},\PT}(t,-q).
\end{split}
\end{align}
In particular, Conjecture \ref{K-conj intro} and compatibility of signs imply\,\footnote{Compatibility of signs means that there exist choices of signs in Conjecture \ref{K-conj intro} compatible with the choices of signs stated in this theorem. For all cases where we checked Conjecture \ref{K-conj intro} (listed in Proposition \ref{verif}), the sign formulae in Remark \ref{sign expec} satisfy this compatibility.}
$$
 \mathsf{V}_{\lambda\mu\varnothing}^{\mathrm{3D},\DT}(t,q)  =  \mathsf{V}_{\lambda\mu\varnothing}^{\mathrm{3D},\PT}(t,q)    \, \mathsf{V}_{\varnothing\varnothing\varnothing}^{\mathrm{3D},\DT}(t,q). 
$$
\end{theorem}

\begin{remark} \label{compatdimredsgns}
In all the cases for which we checked Conjecture \ref{K-conj intro} (see Proposition \ref{verif}), we verified that the compatible choice of signs mentioned in Theorem \ref{dimred intro} exists. This explains our sign choice for the $T$-fixed points which are scheme theoretically supported on $\{x_4=0\}$.
\end{remark}

\begin{theorem} \label{dimredcor}
Assume Conjecture \ref{K-conj intro} and compatibility of signs. Let $D$ be a smooth toric 3-fold and $\beta \in H_2(D,\Z)$ such that all $(\C^*)^3$-fixed points of $\bigcup_n I_n(D,\beta)$, $\bigcup_n P_n(D,\beta)$ have at most two legs in each maximal $(\C^*)^3$-invariant affine open subset of $D$, e.g.~$D$ is a local toric curve or local toric surface. Then the $K$-theoretic DT/PT correspondence \cite[Eqn.~(16)]{NO} holds:
$$
\frac{\sum_{n} \chi(I_n(D,\beta), \widehat{\O}^{\vir}_I) \, q^n}{\sum_{n} \chi(I_n(D,0), \widehat{\O}^{\vir}_I) \, q^n} = \sum_{n} \chi(P_n(D,\beta), \widehat{\O}^{\vir}_P) \, q^n.
$$
\end{theorem}

\begin{remark}
The usual DT/PT correspondence on toric 3-folds \cite{PT2} is a special case of the $K$-theoretic version of Nekrasov-Okounkov \cite[Eqn.~(16)]{NO}.
To the authors' knowledge, the latter is still an open conjecture.
\end{remark}

\subsection{Cohomological limit I}

Let $t_i = e^{b \lambda_i}$, for all $i=1,2,3,4$, and $y = e^{b m}$. We impose the Calabi-Yau relation $t_1t_2t_3t_4=1$, which translates into $\lambda_1+\lambda_2+\lambda_3+\lambda_4=0$. In Section \ref{coho limi 1}, we study the limit $b \rightarrow 0$.
Let $X$ be a Calabi-Yau 4-fold, $\beta \in H_2(X,\Z)$, and $L$ a $T$-equivariant line bundle on $X$. Define the following invariants 
\begin{align}
\begin{split} \label{defcohoinvI}
I_{n,\beta}^{\mathrm{coho}}(L,m) := \sum_{Z \in I_n(X,\beta)^T} &(-1)^{o(\L)|_Z} \frac{\sqrt{(-1)^{\frac{1}{2}\mathrm{ext}^{2}(I_Z,I_Z)}e\big(\Ext^{2}(I_Z,I_Z)\big)}}{e\big(\Ext^{1}(I_Z,I_Z)\big)} \\
&\cdot e(R\Gamma(X, L\otimes \O_Z)^\vee \otimes e^m),
\end{split}
\end{align}
where $\mathrm{ext}^2(I_Z,I_Z) = \dim \Ext^2(I_Z,I_Z)$. The expression under the square root is a square by $T$-equivariant Serre duality. As in Definition \ref{Nekgen}, for a fixed $Z$, two choices of square root differ by a sign and this indeterminacy is absorbed by the choice of orientation $(-1)^{o(\L)|_Z}$.
These invariants take values in 
$$
\frac{\Q(\lambda_1,\lambda_2,\lambda_3,\lambda_4,m)}{(\lambda_1+\lambda_2+\lambda_3+\lambda_4)},
$$
where $\lambda_i := c_1(t_i)$, $m := c_1(e^m)$ denote the $T \times \C^*$-equivariant parameters. Here $\C^*$ corresponds to a trivial torus action with equivariant parameter $e^m$. We similarly define invariants $P_{n,\beta}^{\mathrm{coho}}(L,m)$ replacing $I_n(X,\beta)$ by $P_n(X,\beta)$ and $R\Gamma(X, L\otimes \O_Z)$ by $R\Gamma(X, L\otimes F)$, in which case we also require Assumption \ref{assumption} holds.
\begin{theorem} \label{DT/PT tauto}
Let $X$ be a toric Calabi-Yau 4-fold, $\beta \in H_2(X,\Z)$, and let $L$ be a $T$-equivariant line bundle on $X$. Then
\begin{align*}
\lim_{b \rightarrow 0} \Big( \sum_{n}  I_{n,\beta}(L,y) \, q^n \Big) \Big|_{t_i = e^{b \lambda_i},y = e^{bm}}   &= \sum_{n} I_{n,\beta}^{\mathrm{coho}}(L,m) \, q^n, \\
\lim_{b \rightarrow 0} \Big( \sum_{n}  P_{n,\beta}(L,y) \, q^n \Big) \Big|_{t_i = e^{b \lambda_i},y=e^{bm}} &= \sum_{n} P_{n,\beta}^{\mathrm{coho}}(L,m) \, q^n,
\end{align*}
where the choice of signs on RHS is determined by the choice of signs on LHS. For the second equality, we assume $\bigcup_n P_n(X,\beta)^{(\C^*)^4}$ is at most 0-dimensional. Hence, Conjecture \ref{K-conj intro} implies that there exist choices of signs such that 
$$
\frac{\sum_{n} I_{n,\beta}^{\mathrm{coho}}(L,m) \, q^n}{\sum_{n} I_{n,0}^{\mathrm{coho}}(L,m) \, q^n}
=\sum_{n} P_{n,\beta}^{\mathrm{coho}}(L,m) \, q^n.
$$
\end{theorem}
This theorem provides motivation for conjecturing the following new cohomological DT/PT correspondence for smooth projective Calabi-Yau 4-folds:
\begin{conjecture}
Let $X$ be a smooth projective Calabi-Yau 4-fold and $\beta\in H_2(X,\mathbb{Z})$. For any line bundle $L$ on $X$, there exist choices of orientations such that
$$\frac{\sum_{n}  \int_{[I_n(X,\beta)]^{\vir}}e(L^{[n]}) \, q^n}{\sum_n \int_{[I_n(X,0)]^{\vir}}e(L^{[n]}) \, q^n}= \sum_n \int_{[P_n(X,\beta)]^{\vir}}e(L^{[n]}) \, q^n.$$
\end{conjecture}

\subsection{Cohomological limit II} 

Let $t_i = e^{b \lambda_i}$, $y = e^{b m}$, $Q = m q$, where we again impose the Calabi-Yau relation $t_1t_2t_3t_4=1$. In Section \ref{coho limi 2}, we consider the limit $b \rightarrow 0, m \rightarrow \infty$. In \cite{CK2}, the two first-named authors studied the following cohomological invariants 
\begin{align}
\begin{split} \label{defcohoinvII}
I_{n,\beta}^{\mathrm{coho}} := \sum_{Z \in I_n(X,\beta)^T}
&(-1)^{o(\L)|_Z} \frac{\sqrt{(-1)^{\frac{1}{2}\mathrm{ext}^{2}(I_Z,I_Z)}e\big(\Ext^{2}(I_Z,I_Z)\big)}}{e\big(\Ext^{1}(I_Z,I_Z)\big)}, 
\end{split}
\end{align}
and similar invariants $P_{n,\beta}^{\mathrm{coho}}$, where we replace $I_{n}(X,\beta)$ by $P_n(X,\beta)$ and impose Assumption \ref{assumption}.
In \cite{CK2}, a vertex formalism for these invariants was established giving rise to the cohomological DT/PT vertex
$$
\mathsf{V}_{\lambda\mu\nu\rho}^{\mathrm{coho}, \DT}(Q), \quad \mathsf{V}_{\lambda\mu\nu\rho}^{\mathrm{coho}, \PT}(Q) \in \frac{\Q(\lambda_1,\lambda_2,\lambda_3,\lambda_4)}{(\lambda_1+\lambda_2+\lambda_3+\lambda_4)}(\!(Q)\!),
$$
for any finite plane partitions $\lambda, \mu, \nu, \rho$. As above, in the stable pairs case we assume at most two of these partitions are non-empty. 

The cohomological DT/PT 4-fold vertex correspondence \cite{CK2} states:
\begin{conjecture}[Cao-Kool] \label{conjCK2}
For any finite plane partitions $\lambda, \mu, \nu, \rho$, at most two of which are non-empty, there are choices of signs such that
$$
\mathsf{V}_{\lambda\mu\nu\rho}^{\mathrm{coho}, \DT}(Q) = \mathsf{V}_{\lambda\mu\nu\rho}^{\mathrm{coho}, \PT}(Q) \, \mathsf{V}_{\varnothing\varnothing\varnothing\varnothing}^{\mathrm{coho},\DT}(Q). 
$$ 
\end{conjecture}

\begin{theorem} \label{coholimit intro}
Let $X$ be a toric Calabi-Yau 4-fold and $\beta \in H_2(X,\Z)$. Then
\begin{align*}
\lim_{b \rightarrow 0 \atop m \rightarrow \infty}\Big(\sum_{n}  I_{n,\beta}(\O_X,e^{bm})\, q^n\Big)\Big|_{t_i = e^{b \lambda_i}, Q=qm} &= \sum_{n} I_{n,\beta}^{\mathrm{coho}} \, Q^n, \\
\lim_{b \rightarrow 0 \atop m \rightarrow \infty}\Big(\sum_{n}  P_{n,\beta}(\O_X,e^{bm})\, q^n\Big)\Big|_{t_i = e^{b \lambda_i}, Q=qm} &= \sum_{n} P_{n,\beta}^{\mathrm{coho}} \, Q^n,
\end{align*}
where the choice of signs on RHS is determined by the choice of signs on LHS. For the second equality, we assume $\bigcup_n P_n(X,\beta)^{(\C^*)^4}$ is at most 0-dimensional. Moreover, Conjecture \ref{K-conj intro} implies Conjecture \ref{conjCK2}.
\end{theorem}
We summarise the above three limits in the following figure.
\begin{figure}[htb] 
\xymatrix@!0@R=20pt@C=25.5pt{
 & &&&& & &&   \fbox{$K$-theoretic DT/PT on toric CY 4-fold $X$ }\ar[ddddllll]_(0.5){ \mathrm{Thm}.\, \ref{dimredcor} \quad } 
\ar[ddddrrrr]^(0.6){\begin{subarray}{c} \quad t_i=e^{b\lambda_i},\,\,y=e^{bm},\,\,b\to 0 \\ \\ \mathrm{Thm}.\,\ref{DT/PT tauto} \end{subarray}} \quad \quad &&&&&&&&&&&   \\  \\  \\ \\
&&&&  \fbox{\parbox{109pt}{$K$-theoretic DT/PT \\ on toric 3-fold $D$ \small{\cite{NO}}}}\ar[dddd]_(0.47){\,\,t_i=e^{b\lambda_i},\,\, b\to 0 \,\,}  && &&&& &&
\fbox{\parbox{120pt}{Cohomological DT/PT \\ with insertions on $X$  }} 
\ar[dddd]^(0.46){\begin{subarray}{c} V=\O_X, \,\, Q=qm \\ \\ m\to \infty \\ \\ \mathrm{Thm}.\, \ref{coholimit intro} \end{subarray}}  
\ar[dddllllll]^(0.73){\quad \begin{subarray}{c}\mathrm{Rmk.\, \ref{cohoarrow}}\end{subarray}} 
\\   \\  \\ && &&&&&& \\
&&&& \fbox{\parbox{120pt}{Cohomological DT/PT \\ on toric 3-fold $D$ \small{\cite{PT2}}}}  &&&&& &&&
\fbox{\parbox{120pt}{Cohomological DT/PT \\ without insertions on $X$ \small{\cite{CK2}}}}
%\ar[llllllll]_(0.5){\, \begin{subarray}{c}\mathrm{Thm.\, \ref{dimred intro},\, \ref{coholimit intro}}\\ +\\ \mathrm{vertex\,\, formalism}  \\ {} \end{subarray}}  
 \\   
}
\caption{Limits of $K$-theoretic DT/PT on toric CY 4-folds}
\label{dc1fig1}
\end{figure}
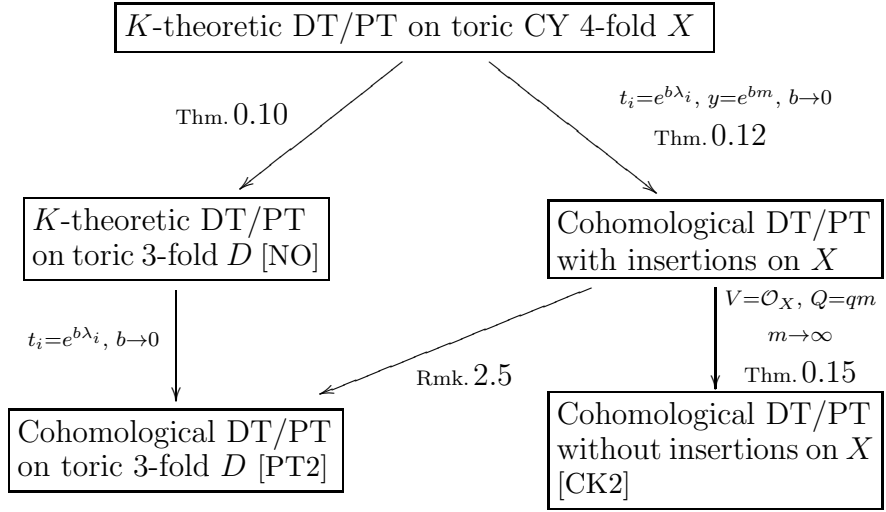

\subsection{Application: local resolved conifold}

In order to illustrate the 4-fold vertex formalism and the three limits, we present a new conjectural formula, which can be seen as a curve analogue of Nekrasov's conjecture. Let $X = D \times \mathbb{C}$, where $D = \mathrm{Tot}_{\mathbb{P}^1}(\O(-1) \oplus \O(-1))$ is the resolved conifold. Consider the generating series of $K$-theoretic stable pair invariants of $X$:  
$$
\cZ_X(y,q,Q) := \sum_{n,d} P_{n,d[\PP^1]}(\O, y)\, q^n Q^d.  
$$
\begin{conjecture} \label{localrescon}
Let $X = \mathrm{Tot}_{\PP^1}(\O(-1) \oplus \O(-1) \oplus \O)$. Then there exist unique choices of signs such that
%\,\footnote{The plethystic exponential is applied to all variables $t_1,t_2,t_3,t_4,y,q,Q$ as in \eqref{defExp}.}
$$
\cZ_X(y,q,Q) = \mathrm{Exp}\Big( \mathcal{F}(t,y;q,Q) \Big), \quad \mathcal{F}(t,y;q,Q) := \frac{Q\,[y]}{[t_4][y^{\frac{1}{2}} q] [y^{\frac{1}{2}} q^{-1}]},
$$
where $t_4^{-1}$ denotes the torus weight of $\O$ over $\PP^1$ and $\mathcal{F}(t,y;q,Q) \in \Q(t_4^{\frac{1}{2}},y^{\frac{1}{2}},q,Q)$ is expanded as a formal power series in $q$ and $Q$.
\end{conjecture}

This conjecture is verified modulo (more or less) $Q^5 q^6$ using the vertex formalism. 
See Proposition \ref{verif localrescon} for the precise statement. 
Applying dimensional reduction, Conjecture \ref{localrescon} implies a formula for the 
$K$-theoretic stable pair invariants of the resolved conifold $D$ recently proved by Kononov-Okounkov-Osinenko \cite{KOO}. Applying the preferred limits discussed by Arbesfeld \cite{Arb} to the formula of Kononov-Okounkov-Osinenko yields an expression obtained using the refined topological vertex by Iqbal-Koz{\c{c}}az-Vafa \cite{IKV}. Applying cohomological limit II yields a formula, which was recently conjectured in \cite{CK2}. See Appendix \ref{app:localrescon} for the details.

\subsection{Relations with other works} \label{otherwork}
This paper is a continuation of our previous work \cite{CK2}, where we introduced the DT/PT correspondence (with primary insertions) for both compact and toric Calabi-Yau 4-folds. In the compact case, ``DT=PT'' due to insertions. In loc.~cit.~we used toric calculations to support this result and found the cohomological DT/PT 4-fold vertex correspondence (Conjecture \ref{conjCK2}), which surprisingly has the same shape as the DT/PT correspondence for Calabi-Yau 3-folds \cite{PT1}. This motivated us to enhance Conjecture \ref{conjCK2} to a $K$-theoretic version using Nekrasov's insertion (Definition \ref{Nekgen}), which specializes to (i) the cohomological DT/PT correspondence for toric Calabi-Yau 4-folds,  (ii) the $K$-theoretic DT/PT correspondence for toric 3-folds \cite{NO, PT2}. 

During the writing of this paper, Piazzalunga announced\footnote{``Gauge theory and virtual invariants'', Trinity College Dublin, May 13--17, 2019.} his joint work with Nekrasov \cite{NP2}, which establishes the (more general) $K$-theoretic DT vertex for D0-D2-D4-D6-D8 bound states and realizes the DT generating series of a toric Calabi-Yau 4-fold $X$ as the partition function of a certain super-Yang-Mills theory with matter (and gauge group $\mathrm{U}(1|1$)) on $X$.

\subsection{Acknowledgements} 

This paper has a strong intellectual debt to the remarkable work of Nekrasov \cite{Nek} and Nekrasov-Piazzalunga \cite{NP} without which this paper would not exist. 
We are very grateful to Nikita Nekrasov and Nicol\`{o} Piazzalunga for correspondence and helpful explanations of \cite{Nek, NP}.
We also warmly thank Noah Arbesfeld for many useful discussions on $K$-theoretic invariants.
Y.C. is partially supported by the World Premier International Research Center Initiative (WPI), MEXT, Japan, 
the JSPS KAKENHI Grant Number JP19K23397 and Royal Society Newton International Fellowships Alumni 2019.
M.K. is supported by NWO grant VI.Vidi.192.012.
S.M. is supported by NWO grant TOP2.17.004.

\section{$K$-theoretic vertex formalism} \label{vertex}

\subsection{Fixed loci} \label{fixlocus}

This subsection is a recap of \cite[Sect.~2.1, 2.2]{CK2}. Proofs can be found in loc.~cit. Let $X$ be a toric Calabi-Yau 4-fold. Let $\Delta(X)$ be the polytope corresponding to $X$ and denote the collection of its vertices by $V(X)$ and edges by $E(X)$. The elements $\alpha \in V(X)$ correspond to the $(\C^*)^4$-fixed points $p_\alpha \in X$. Each such fixed point lies in a maximal $(\C^*)^4$-invariant affine open subset $\C^4 \cong U_\alpha \subseteq X$. The elements $\alpha\beta \in E(X)$ (connecting vertices $\alpha$, $\beta$) correspond to the $(\C^*)^4$-invariant lines $\PP^1 \cong L_{\alpha\beta} \subseteq X$ with normal bundle
\begin{align}
\begin{split} \label{normal}
&N_{L_{\alpha\beta}/X} \cong \O_{\PP^1}(m_{\alpha\beta}) \oplus  \O_{\PP^1}(m_{\alpha\beta}') \oplus  \O_{\PP^1}(m_{\alpha\beta}''), \\
&m_{\alpha\beta} + m_{\alpha\beta}' + m_{\alpha\beta}'' = -2, 
\end{split}
\end{align}
where the second equality follows from the Calabi-Yau condition.

The action of the dense open torus $(\C^*)^4$ and its Calabi-Yau subtorus $T \subseteq (\C^*)^4$ (the subtorus preserving the Calabi-Yau volume form) both lift to the Hilbert scheme $I:=I_n(X,\beta)$. The following result is proved in \cite[Lem.~2.2]{CK2}:
\begin{lemma} \label{TfixlocusDT}
The scheme $I^T=I^{(\mathbb{C}^*)^4}$ consists of finitely many reduced points.
\end{lemma}

We characterize the elements $I^{T}$ by collections of solid partitions. 
\begin{definition} \label{solid}
A solid partition $\pi$ is a sequence $\pi = \big\{\pi_{ijk} \in \Z_{ \geqslant 0} \cup \{\infty\} \big\}_{i,j,k \geqslant 1}$ satisfying:
$$
\pi_{ijk} \geqslant \pi_{i+1,j,k}, \qquad \pi_{ijk} \geqslant \pi_{i,j+1,k}, \qquad \pi_{ijk} \geqslant \pi_{i,j,k+1} \qquad \forall \,\, i,j,k \geqslant 1.
$$
This extends the notions of \emph{plane partitions} $\lambda = \{ \lambda_{ij}\}_{i,j \geqslant 1}$ (which we visualize as a pile of boxes in $\R^3$ where $\lambda_{ij}$ is the height along the $x_3$-axis) and \emph{line partitions} $\lambda = \{ \lambda_{i}\}_{i \geqslant 1}$ (which we visualize as a pile of squares in $\R^2$ where $\lambda_i$ is the height along the $x_2$-axis). Given a solid partition $\pi = \{\pi_{ijk}\}_{i,j,k \geqslant 1}$, there exist unique plane partitions $\lambda, \mu, \nu, \rho$ such that
\begin{align*}
&\pi_{ijk} = \lambda_{jk}, \qquad \forall \,\, i \gg 0,\,\, j,k \geqslant 1 \\
&\pi_{ijk} = \mu_{ik}, \qquad \forall \,\, j \gg 0, \,\, i,k \geqslant 1 \\
&\pi_{ijk} = \nu_{ij}, \,\qquad \forall \,\, k \gg 0,\,\,  i,j \geqslant 1 \\
&\pi_{ijk} = \infty \Leftrightarrow \,\, k = \rho_{ij}, \quad \forall \,\, i,j,k \geqslant 1.
\end{align*}
We refer to $\lambda, \mu, \nu, \rho$ as the \emph{asymptotic plane partitions} associated to $\pi$ in directions $1,2,3,4$ respectively. We call $\pi$ \emph{point-like}, when $\lambda = \mu = \nu = \rho =  \varnothing$. Then the \emph{size} of $\pi$ is defined by
$$
|\pi| := \sum_{i,j,k \geqslant 1} \pi_{ijk}.
$$
We call $\pi$ \emph{curve-like} when $\lambda, \mu, \nu, \rho$ have finite size $|\lambda|, |\mu|, |\nu|, |\rho|$ (not all zero). Similar to \cite{MNOP}, when $\pi$ is curve-like, we define its \emph{renormalized volume} by
$$
|\pi| := \sum_{(i,j,k,l) \in \mathbb{Z}_{\geqslant 1}^4 \atop l \leqslant \pi_{ijk}} \big( 1 - \#\{\textrm{legs containing } (i,j,k,l) \} \big).
%|\pi| := \sum_{1 \leqslant i,j,k \leqslant N} \pi_{ijk} - (|\lambda| + |\mu| + |\nu|) \cdot N,
$$
%which is independent of $N \gg 0$.
\end{definition}

Let $Z \in I^{T}$. Suppose $\C^4 \cong U_{\alpha} \subseteq X$ is a maximal $(\C^*)^4$-invariant affine open subset. There are coordinates $(x_1,x_2,x_3,x_4)$ on $U_{\alpha}$ such that
\begin{equation} \label{standardaction} 
t \cdot x_i = t_i x_i, \quad \textrm{for all } i=1,2,3,4 \, \textrm{ and } \, t = (t_1,t_2,t_3,t_4) \in (\C^*)^4. 
\end{equation} 
The restriction $Z|_{U_{\alpha}}$ is given by a $(\C^*)^4$-invariant ideal $I_Z|_{U_{\alpha}} \subseteq \C[x_1,x_2,x_3,x_4]$. Solid partitions $\pi$ which are point- or curve-like are in bijective correspondence to $(\C^*)^4$-invariant ideals $I_{Z_\pi} \subseteq \C[x_1,x_2,x_3,x_4]$ cutting out subschemes $Z_\pi \subseteq \C^4$ of dimension $\leqslant 1$ via the following formula
\begin{equation} \label{partitionideal}
I_{Z_\pi} = \left( x_1^{i-1} x_2^{j-1} x_3^{k-1} x_4^{\pi_{ijk}} \, : \, i,j,k \geqslant 1 \textrm{ \ such that \ } \pi_{ijk} < \infty   \right).
\end{equation}
Hence $Z \in I^{T}$ determines a collection of (point- or curve-like) solid partitions $\{\pi^{(\alpha)}\}_{\alpha = 1}^{e(X)}$, where $e(X)$ is the topological Euler characteristic of $X$, i.e.~the number of $(\C^*)^4$-fixed points of $X$. Let $\alpha\beta \in E(X)$ and consider the corresponding $(\C^*)^4$-invariant line $L_{\alpha\beta} \cong \PP^1$. Suppose this line given by $\{x_2=x_3=x_4=0\}$ in both charts $U_\alpha$, $U_\beta$. Let $\lambda^{(\alpha)}$, $\lambda^{(\beta)}$ be the asymptotic plane partitions of $\pi^{(\alpha)}$, $\pi^{(\beta)}$ along the $x_1$-axes in both charts. Then
\begin{equation} \label{glue}
\lambda^{(\alpha)}_{ij} = \lambda^{(\beta)}_{ij} =: \lambda_{ij}^{(\alpha\beta)} \qquad \forall \,\, i,j \geqslant 1. 
\end{equation}
A collection of point- or curve-like solid partitions $\{\pi^{(\alpha)}\}_{\alpha = 1}^{e(X)}$ satisfying \eqref{glue}, for all $\alpha, \beta = 1, \ldots, e(X)$, is said to satisfy the \emph{gluing condition}. We obtain a bijective correspondence  
\begin{align*}
&\left\{ {\boldsymbol{\pi}} = \{\pi^{(\alpha)}\}_{\alpha = 1}^{e(X)} \, : \,  \pi^{(\alpha)} \textrm{ \, point- or curve-like and satisfying \eqref{glue}} \right\} \\
&\qquad \qquad \qquad \stackrel{1-1}{\longleftrightarrow}  \\  
&Z_{{\boldsymbol{\pi}}} \in \bigcup_{\beta \in H_2(X,\Z), \, n \in \Z} I_n(X,\beta)^{T}.
\end{align*}

Suppose $\beta \neq 0$ is effective. Then for any $Z \in I^T = I_n(X,\beta)^{T}$, there exists a maximal Cohen-Macaulay subscheme $C \subseteq Z$ such that the cokernel
$$
0 \rightarrow I_Z \rightarrow I_C \rightarrow I_C / I_Z \rightarrow 0
$$
is 0-dimensional. The restriction $C|_{U_\alpha}$ is empty or corresponds to a curve-like solid partition $\pi$ with asymptotics $\lambda, \mu, \nu, \rho$. Since $C|_{U_\alpha}$ has no embedded points, the solid partition $\pi$ is entirely determined by the asymptotics $\lambda, \mu ,\nu, \rho$ as follows
\begin{equation} \label{CMsolid}
\pi_{ijk} = \left\{\begin{array}{cc} \infty & \textrm{if \, } 1 \leqslant k \leqslant \rho_{ij} \\ \max\{ \lambda_{jk}, \mu_{ik}, \nu_{ij} \}  & \textrm{otherwise.}  \end{array} \right.
\end{equation}

Using similar notation to \cite{MNOP}, for any plane partition of finite size, and $m,m',m'' \in \Z$, we define
$$
f_{m,m',m''}(\lambda) := \sum_{i,j \geqslant 1} \sum_{k=1}^{\lambda_{ij}} (1-m(i-1) - m'(j-1) - m''(k-1)).
$$
For any $\alpha\beta \in E(X)$ and finite plane partition $\lambda$, we define
\begin{equation} \label{fab}
f(\alpha,\beta) := f_{m_{\alpha\beta},m_{\alpha\beta}',m_{\alpha\beta}''}(\lambda),
\end{equation}
where $m_{\alpha\beta}, m_{\alpha\beta}', m_{\alpha\beta}''$ were defined in \eqref{normal}.
\begin{lemma} \label{chi}
Let $X$ be a toric Calabi-Yau 4-fold and let $Z \subseteq X$ be a $(\C^*)^4$-invariant closed subscheme of dimension $\leqslant 1$. Then
$$
\chi(\O_Z)  = \sum_{\alpha \in V(X)} |\pi^{(\alpha)}| + \sum_{\alpha\beta \in E(X)} f(\alpha,\beta).
$$
\end{lemma}

The action of $(\C^*)^4$ on $X$ also lifts to the moduli space $P:=P_n(X,\beta)$ of stable pairs. Similar to \cite{PT2}, we give a description of the fixed locus $P^{(\C^*)^4}$. 

For any stable pair $(F,s)$ on $X$, the scheme-theoretic support $C_F := \mathrm{supp}(F)$ is a Cohen-Macaulay curve \cite[Lem.~1.6]{PT1}. Stable pairs with Cohen-Macaulay support curve $C$ can be described as follows \cite[Prop.~1.8]{PT1}: \\

\noindent Let $\mathfrak{m} \subseteq \O_C$ be the ideal of a finite union of closed points on $C$. A stable pair $(F,s)$ on $X$ such that $C_F = C$ and $\mathrm{supp}(Q)_{\mathrm{red}} \subseteq \mathrm{supp}(\O_C / \mathfrak{m})$ is equivalent to a subsheaf of $\varinjlim \hom(\mathfrak{m}^r, \O_C) / \O_C$. \\

This uses the natural inclusions
\begin{align*}
\hom(\mathfrak{m}^r, \O_C) &\hookrightarrow \hom(\mathfrak{m}^{r+1}, \O_C) \\
\O_C &\hookrightarrow \hom(\mathfrak{m}^r, \O_C)
\end{align*}
induced by $\mathfrak{m}^{r+1} \subseteq \mathfrak{m}^{r} \subseteq \O_C$. 

Suppose $[(F,s)] \in P^{(\C^*)^4}$, then $C_F$ is $(\C^*)^4$-fixed and determines $\{\pi^{(\alpha)}\}_{\alpha \in V(X)}$ with each $\pi^{(\alpha)}$ empty or a curve-like solid partition. Consider a maximal $(\C^*)^4$-invariant affine open subset $\C^4 \cong U_\alpha \subseteq X$. Denote the asymptotic plane partitions of $\pi := \pi^{(\alpha)}$ in directions $1,2,3,4$ by $\lambda, \mu, \nu, \rho$. As in \eqref{partitionideal}, these correspond to $(\C^*)^4$-invariant ideals
\begin{align*}
I_{Z_\lambda} &\subseteq \C[x_2,x_3,x_4], \\
I_{Z_\mu} &\subseteq \C[x_1,x_3,x_4], \\
I_{Z_\nu} &\subseteq \C[x_1,x_2,x_4], \\
I_{Z_\rho} &\subseteq \C[x_1,x_2,x_3].
\end{align*}
Define the following $\C[x_1,x_2,x_3,x_4]$-modules
\begin{align*}
M_1 &:= \C[x_1,x_1^{-1}] \otimes_{\C} \C[x_2,x_3,x_4] / I_{Z_\lambda}, \\
M_2 &:= \C[x_2,x_2^{-1}] \otimes_{\C} \C[x_1,x_3,x_4] / I_{Z_\mu}, \\
M_3 &:= \C[x_3,x_3^{-1}] \otimes_{\C} \C[x_1,x_2,x_4] / I_{Z_\nu}, \\
M_4 &:= \C[x_4,x_4^{-1}] \otimes_{\C} \C[x_1,x_2,x_3] / I_{Z_\rho}.
\end{align*}
Then \cite[Sect.~2.4]{PT2} gives 
$$
\varinjlim \hom(\mathfrak{m}^r, \O_{C|_{U_\alpha}})  \cong \bigoplus_{i=1}^4 M_i =: M,
$$
where $\mathfrak{m} = ( x_1, x_2, x_3, x_4 ) \subseteq \C[x_1,x_2,x_3,x_4]$. Each module $M_i$ comes from a ring, so it has a unit 1, which is homogeneous of degree $(0,0,0,0)$ with respect to the character group $X((\C^*)^4) = \Z^4$. We consider the quotient
\begin{equation} \label{Mmod}
M / \langle (1,1,1,1) \rangle.
\end{equation}
Then $(\C^*)^4$-equivariant stable pairs on $U_\alpha \cong \C^4$ correspond to $(\C^*)^4$-invariant $\C[x_1,x_2,x_3,x_4]$-submodules of \eqref{Mmod}. \\

\noindent \textbf{Combinatorial description of $M / \langle (1,1,1,1) \rangle$} Denote the character group of $(\C^*)^4$ by $X((\C^*)^4) = \Z^4$. For each module $M_i$, the weights $w \in \Z^4$ of its non-zero eigenspaces determine an infinite ``leg'' $\mathrm{Leg}_i \subseteq \Z^4$ along the $x_i$-axis. For each weight $w \in \Z^4$, introduce four independent vectors $\boldsymbol{1}_w$, $\boldsymbol{2}_w$, $\boldsymbol{3}_w$, $\boldsymbol{4}_w$. Then the $\C[x_1,x_2,x_3,x_4]$-module structure on $M / \langle (1,1,1,1) \rangle$ is determined by
$$
x_j \cdot \boldsymbol{i}_{w} = \boldsymbol{i}_{w+e_j},
$$
where $i,j=1, 2, 3, 4$ and $e_1, e_2, e_3, e_4$ are the standard basis vectors of $\Z^4$. Similar to the 3-fold case \cite[Sect.~2.5]{PT2}, we define regions
$$
\mathrm{I}^+ \cup \mathrm{II} \cup \mathrm{III} \cup \mathrm{IV} \cup \mathrm{I}^- = \bigcup_{i=1}^4 \mathrm{Leg}_i  \subseteq \Z^4, \quad\textrm{where}
$$
\begin{itemize}
\item $\mathrm{I}^+$ consists of the weights $w \in \Z^4$ with all coordinates non-negative \emph{and} which lie in precisely one leg. If $w \in \mathrm{I}^+$, then the corresponding weight space of $M / \langle (1,1,1,1) \rangle$ is 0-dimensional.
\item $\mathrm{I}^-$ consists of all weights $w \in \Z^4$ with at least one negative coordinate. If $w \in \mathrm{I}^-$ is supported in $\mathrm{Leg}_i$, then the corresponding weight space of $M / \langle (1,1,1,1) \rangle$ is 1-dimensional
$$
\C \cong \C \cdot \boldsymbol{i}_w \subseteq M / \langle (1,1,1,1) \rangle.
$$
\item $\mathrm{II}$ consists of all weights $w \in \Z^4$, which lie in precisely two legs. If $w \in \mathrm{II}$ is supported in $\mathrm{Leg}_i$ and $\mathrm{Leg}_j$, then the corresponding weight space of $M / \langle (1,1,1,1) \rangle$ is 1-dimensional
$$
\C \cong \C \cdot \boldsymbol{i}_w \oplus \C \cdot \boldsymbol{j}_w / \C \cdot (\boldsymbol{i}_w + \boldsymbol{j}_w) \subseteq M / \langle (1,1,1,1) \rangle.
$$
\item $\mathrm{III}$ consists of all weights $w \in \Z^4$, which lie in precisely three legs. If $w \in \mathrm{III}$ is supported in $\mathrm{Leg}_i$, $\mathrm{Leg}_j$, and $\mathrm{Leg}_k$, then the corresponding weight space of $M / \langle (1,1,1,1) \rangle$ is 2-dimensional
$$
\C^2 \cong \C \cdot \boldsymbol{i}_w \oplus \C \cdot \boldsymbol{j}_w \oplus \C \cdot \boldsymbol{k}_w / \C \cdot (\boldsymbol{i}_w + \boldsymbol{j}_w + \boldsymbol{k}_w) \subseteq M / \langle (1,1,1,1) \rangle.
$$
\item $\mathrm{IV}$ consists of all weights $w \in \Z^4$, which lie in all four legs. If $w \in \mathrm{IV}$, then the corresponding weight space of $M / \langle (1,1,1,1) \rangle$ is 3-dimensional
$$
\C^3 \cong \C \cdot \boldsymbol{1}_w \oplus \C \cdot \boldsymbol{2}_w \oplus \C \cdot \boldsymbol{3}_w \oplus \C \cdot \boldsymbol{4}_w / \C \cdot (\boldsymbol{1}_w + \boldsymbol{2}_w + \boldsymbol{3}_w + \boldsymbol{4}_w) \subseteq M / \langle (1,1,1,1) \rangle.
$$
\end{itemize}

\noindent \textbf{Box configurations}.  A box configuration is a finite collection of weights $B \subseteq \mathrm{II} \cup \mathrm{III} \cup \mathrm{IV} \cup \mathrm{I}^-$ satisfying the following property: \\

\noindent if $w = (w_1,w_2,w_3,w_4) \in \mathrm{II} \cup \mathrm{III} \cup \mathrm{IV} \cup \mathrm{I}^-$ and one of $(w_1-1,w_2,w_3,w_4)$, $(w_1,w_2-1,w_3,w_4)$, $(w_1,w_2,w_3-1,w_4)$, or $(w_1,w_2,w_3,w_4-1)$ lies in $B$ then $w \in B$. \\

A box configuration determines a $(\C^*)^4$-invariant submodule of $M / \langle (1,1,1,1) \rangle$ and therefore a $(\C^*)^4$-invariant stable pair on $U_\alpha \cong \C^4$ with cokernel of length
$$
\#  (B \cap \mathrm{II}) + 2 \cdot \#  (B \cap \mathrm{III}) + 3 \cdot \#  (B \cap \mathrm{IV}) + \#  (B \cap \mathrm{I}^-).
$$
The box configurations defined in this section do \emph{not} describe all $(\C^*)^4$-invariant submodules of $M / \langle (1,1,1,1) \rangle$. In this paper, we always work with Assumption \ref{assumption} from the introduction, i.e.~$\bigcup_n P_n(X,\beta)^{(\C^*)^4}$ is at most 0-dimensional. Then the restriction of any $T$-fixed stable pair $(F,s)$ on $X$ to any chart $U_\alpha$ has a Cohen-Macaulay support curve with at most two asymptotic plane partitions and is described by a box configuration as above. See \cite[Prop.~2.5, 2.6]{CK2}:

\begin{lemma} \label{fewlegs}
Suppose $\bigcup _n P_n(X,\beta)^{(\C^*)^4}$ is at most 0-dimensional. Then for any $[(F,s)] \in P_n(X,\beta)^{(\C^*)^4}$ and any $\alpha \in V(X)$, the Cohen-Macaulay curve $C_F |_{U_\alpha}$ has at most two asymptotic plane partitions.
\end{lemma}

\begin{lemma}  \label{PTfixedlocus}
Suppose $\bigcup_n P_n(X,\beta)^{(\C^*)^4}$ is at most 0-dimensional. Then, for any $n \in \Z$, $P_n(X,\beta)^{T} = P_n(X,\beta)^{(\C^*)^4}$ consists of finitely many reduced points.
\end{lemma}

\subsection{$K$-theory class of obstruction theory} 

Let $X$ be a toric Calabi-Yau 4-fold and consider the cover $\{U_\alpha\}_{\alpha \in V(X)}$ by maximal $(\C^*)^4$-invariant affine open subsets. We discuss the DT and PT case simultaneously. Let $E = I_Z$, with $Z\in I^T = I_n(X,\beta)^T$, or $E = I^\mdot$, with $[I^\mdot = \{\O_X \rightarrow F\} ] \in P^T = P_n(X,\beta)^T$. In the stable pairs case, we impose Assumption \ref{assumption} of the introduction so $P^T = P^{(\C^*)^4}$ is at most 0-dimensional by Proposition \ref{PTfixedlocus}. We are interested in the class
$$
-\dR\Hom(E,E)_0 \in K_0^{(\mathbb{C}^*)^4}(\mathrm{pt}).
$$
Note that in this section, we work with the \emph{full} torus $(\mathbb{C}^*)^4$. In the next section, we will restrict to the Calabi-Yau torus $T \subseteq (\C^*)^4$ when taking square roots.
This class can be computed by a \v{C}ech calculation introduced for smooth toric 3-folds in \cite{MNOP, PT2}. In the case of toric 4-folds, the calculation was done in \cite[Sect.~2.4]{CK2}. We briefly recall the results from loc.~cit.

Consider the exact triangle
\begin{equation*} 
E \rightarrow \O_X \rightarrow E',
\end{equation*}
where $E' = \O_Z$ when $E = I_Z$, and $E' = F$ when $E = I^\mdot$. In both cases, $E'$ is 1-dimensional. Define $U_{\alpha\beta} := U_\alpha \cap U_\beta$,\, $U_{\alpha\beta\gamma} := U_\alpha \cap U_\beta \cap U_\gamma$, etc., and let $E_\alpha := E|_{U_{\alpha}}$, $E_{\alpha\beta} := E|_{U_{\alpha\beta}}$ etc. 
The local-to-global spectral sequence, calculation of sheaf cohomology with respect to the \v{C}ech cover $\{U_\alpha\}_{\alpha \in V(X)}$, and the fact that $E'$ is 1-dimensional give

\begin{align*}
-\dR\mathrm{Hom}_X(E,E)_0 &= - \sum_{\alpha \in V(X)} \dR\mathrm{Hom}_{U_\alpha}(E_\alpha,E_\alpha)_0 + \sum_{\alpha\beta \in E(X)} \dR\mathrm{Hom}_{U_{\alpha\beta}}(E_{\alpha\beta},E_{\alpha\beta})_0.
\end{align*}

On $U_\alpha \cong \mathbb{C}^4$, we use coordinates $x_1,x_2,x_3, x_4$ such that the $(\mathbb{C}^*)^4$-action is 
$$
t \cdot x_i = t_i x_i, \quad \textrm{for all } i=1,2,3,4 \, \textrm{ and } \, t=(t_1,t_2,t_3,t_4) \in (\C^*)^4. 
$$
Denote the $T$-character of $E'|_{U_\alpha}$ by
$$
Z_\alpha := \tr_{E'|_{U_\alpha}}.
$$
In the case $E' = \O_Z$, the scheme $Z|_{U_\alpha}$ corresponds to a solid partition $\pi^{(\alpha)}$ as described in Section \ref{fixlocus} and
\begin{equation} \label{Zalpha}
Z_\alpha = \sum_{i,j,k\geqslant1} \sum_{l=1}^{\pi_{ijk}^{(\alpha)}} t_1^{i-1} t_2^{j-1} t_3^{k-1} t_4^{l-1}.
\end{equation}
When $E' = F$, we use the short exact sequence
$$
0 \rightarrow \O_C \rightarrow F \rightarrow Q \rightarrow 0,
$$
where $C$ is the Cohen-Macaulay support curve and $Q$ is the cokernel. Then
\begin{equation} \label{PTZalpha}
Z_\alpha = \tr_{\O_C|_{U_\alpha}} + \tr_{Q|_{U_\alpha}},
\end{equation}
where $\O_C|_{U_\alpha}$ is described by a solid partition $\pi^{(\alpha)}$ and $Q|_{U_\alpha}$ is described by a box configuration $B^{(\alpha)}$ as in Section \ref{fixlocus} (by Assumption \ref{assumption}). In this case, $\tr_{\O_C|_{U_\alpha}}$ is given by the RHS of \eqref{Zalpha}. Moreover, $\tr_{Q|_{U_\alpha}}$ is the sum of $t^w$ over all $w \in B^{(\alpha)}$.

For any $\alpha\beta \in E(X)$, we consider
$$
Z_{\alpha\beta} := \tr_{E'|_{U_{\alpha\beta}}}.
$$
In both cases, $E = I_Z$ and $E = I^\mdot$, there is an underlying Cohen-Macaulay curve $C|_{U_{\alpha\beta}}$. Suppose in both charts $U_\alpha, U_\beta$, the line $L_{\alpha\beta} \cong \PP^1$ is given by $\{x_2 = x_3 = x_4 = 0\}$. Note that $U_{\alpha\beta} \cong \C^* \times \C^3$. Then $C|_{U_{\alpha\beta}}$ is described by a point-like plane partition $\lambda_{\alpha\beta}$ (its cross-section along the $x_1$-axis) and
\begin{equation} \label{Zalphabeta}
Z_{\alpha\beta} = \sum_{j,k \geqslant 1} \sum_{l=1}^{\lambda_{\alpha\beta}} t_2^{j-1} t_3^{k-1} t_4^{l-1}
\end{equation}

Using an equivariant resolution of $E_\alpha$, $E_{\alpha\beta}$, one readily obtains the following formulae for the $T$-representations of  $-\dR\mathrm{Hom}(E_\alpha,E_\alpha)_0$, $\dR\mathrm{Hom}(E_{\alpha\beta},E_{\alpha\beta})_0$ (see \cite[Sect.~2.4]{CK2}, which is based on the original calculation in \cite{MNOP})
\begin{align} 
\begin{split} \label{defVEprelim}
\tr_{-\dR\mathrm{Hom}(E_\alpha,E_\alpha)_0} &=  Z_\alpha + \frac{\overline{Z}_{\alpha}}{t_1t_2t_3t_4} - \frac{P_{1234}}{t_1t_2t_3t_4} Z_\alpha \overline{Z}_\alpha, \\
-\tr_{-\dR\mathrm{Hom}(E_{\alpha\beta},E_{\alpha\beta})_0} &=  \delta(t_1) \left( -Z_{\alpha\beta} + \frac{\overline{Z}_{\alpha\beta}}{t_2t_3t_4} - \frac{P_{234}}{t_2t_3t_4} Z_{\alpha\beta} \overline{Z}_{\alpha\beta} \right),
\end{split}
\end{align}
where $\overline{(\cdot)}$ is the involution on $K_0^T(\mathrm{pt})$ mentioned in the introduction and
\begin{align}
\begin{split} \label{defdelta}
\delta(t) &:= \sum_{n \in \Z} t^{n}, \\
P_{1234} &:= (1-t_1)(1-t_2)(1-t_3)(1-t_4), \\
P_{234} &:= (1-t_2)(1-t_3)(1-t_4).
\end{split}
\end{align}
As in \cite{MNOP}, one has to be careful about the precise meaning of \eqref{defVEprelim}. For instance, in the DT case, when $\pi^{(\alpha)}$ is point-like, the formula for $\tr_{-\dR\mathrm{Hom}(E_\alpha,E_\alpha)_0}$ in \eqref{defVEprelim} is  Laurent polynomial in the variables $t_i$. However, when $\pi^{(\alpha)}$ is curve-like, the infinite sum \eqref{Zalpha} for $Z_{\alpha}$ first has to be expressed as a rational function and $\tr_{-\dR\mathrm{Hom}(E_\alpha,E_\alpha)_0}$ is then viewed as a rational function in the variables $t_i$.

The problem with \eqref{defVEprelim} is that it consists of \emph{rational functions} in $t_1,t_2,t_3,t_4$. From \cite{MNOP}, we learn how to redistribute terms in such a ways that we obtain \emph{Laurent polynomials}. Let $\beta_1,\beta_2,\beta_3,\beta_4$ be the vertices neighbouring $\alpha$. Define\,\footnote{We only write down $\mathsf{F}_{\alpha\beta}$ and $\mathsf{E}_{\alpha\beta}$ when $L_{\alpha\beta} \cong \PP^1$ is given by $\{x_2=x_3=x_4=0\}$, i.e.~the leg along the $x_1$-axis. The other cases follow by symmetry.}
\begin{align}
\begin{split} \label{defVEF}
\mathsf{F}_{\alpha\beta} &:= -Z_{\alpha\beta} + \frac{\overline{Z}_{\alpha\beta}}{t_2t_3t_4}  - \frac{P_{234}}{t_2t_3t_4} Z_{\alpha\beta} \overline{Z}_{\alpha\beta}, \\
\mathsf{V}_\alpha &:= \tr_{-\dR\mathrm{Hom}(E_\alpha,E_\alpha)_0} + \sum_{i=1}^{4} \frac{\mathsf{F}_{\alpha\beta_i}(t_{i'},t_{i''},t_{i'''})}{1-t_i}, \\
\mathsf{E}_{\alpha\beta} &:= t_1^{-1} \frac{\mathsf{F}_{\alpha\beta}(t_2,t_3,t_4)}{1-t_1^{-1}} - \frac{\mathsf{F}_{\alpha\beta}(t_2t_1^{-m_{\alpha\beta}},t_3t_1^{-m'_{\alpha\beta}},t_4 t_1^{-m''_{\alpha\beta}})}{1-t_1^{-1}},
\end{split}
\end{align}
where $\{t_{i}, t_{i'}, t_{i''}, t_{i'''}  \} = \{t_1, t_2, t_3, t_4\}$ and $$(t_1,t_2,t_3,t_4) \mapsto (t_1^{-1}, t_2 t_1^{-m_{\alpha\beta}},  t_3 t_1^{-m'_{\alpha\beta}},  t_4 t_1^{-m''_{\alpha\beta}})$$ corresponds to the coordinate transformation $U_\alpha \rightarrow U_\beta$ and $m_{\alpha\beta}, m'_{\alpha\beta}, m''_{\alpha\beta}$ are the weights of the normal bundle of $L_{\alpha\beta}$ defined in \eqref{normal}. Then
$$
\tr_{-\dR\mathrm{Hom}(E,E)_0} = \sum_{\alpha \in V(X)} \mathsf{V}_\alpha + \sum_{\alpha\beta \in E(X)} \mathsf{E}_{\alpha\beta},
$$
and $\mathsf{V}_\alpha$, $\mathsf{E}_{\alpha\beta}$ are Laurent polynomials for all $\alpha \in V(X)$ and $\alpha\beta \in E(X)$ by \cite[Prop.~2.11]{CK2}.

\begin{remark}
When we want to stress the dependence on $Z_\alpha$, $Z_{\alpha\beta}$ and distinguish between the DT/PT case, we write $\mathsf{V}_{Z_\alpha}^{\DT}$, $\mathsf{V}_{Z_\alpha}^{\PT}$, $\mathsf{E}_{Z_{\alpha\beta}}^{\DT}$, $\mathsf{E}_{Z_{\alpha\beta}}^{\PT}$ for the classes introduced in \eqref{defVEF}.
\end{remark}

\subsection{Taking square roots} \label{sec:taking roots} 

Let $\alpha \in V(X)$. As before, denote by $\beta_1, \ldots, \beta_4 \in V(X)$ the vertices neighbouring $\alpha$ and labelled such that $L_{\alpha\beta_i} = \{x_{i'}=x_{i''} = x_{i'''} = 0\}$, where $\{i',i'',i'''\} = \{1,2,3,4\} \setminus \{i\}$. We define
\begin{align}
\begin{split} \label{defvef1} 
\mathsf{v}_\alpha &:= Z_\alpha - \overline{P}_{123} Z_\alpha \overline{Z}_\alpha  + \sum_{i=1}^{3} \frac{\mathsf{f}_{\alpha\beta_i}(t_{i'},t_{i''},t_{i'''})}{1-t_i} \\
&+ \frac{1}{(1-t_4)} \Big\{ - Z_{\alpha\beta_4} + \overline{P}_{123} \big( \overline{Z}_\alpha Z_{\alpha\beta_4} - Z_\alpha \overline{Z}_{\alpha\beta_4} \big) + \frac{\overline{P}_{123}}{1-t_4} Z_{\alpha\beta_4} \overline{Z}_{\alpha\beta_4} \Bigg\}, \\
\mathsf{e}_{\alpha\beta_1} &:= t_1^{-1} \frac{\mathsf{f}_{\alpha\beta_1}(t_2,t_3,t_4)}{1-t_1^{-1}} - \frac{\mathsf{f}_{\alpha\beta_1}(t_2t_1^{-m_{\alpha\beta_1}},t_3t_1^{-m'_{\alpha\beta_1}},t_4 t_1^{-m''_{\alpha\beta_1}})}{1-t_1^{-1}}, \\
\mathsf{e}_{\alpha\beta_2} &:= t_2^{-1} \frac{\mathsf{f}_{\alpha\beta_2}(t_1,t_3,t_4)}{1-t_2^{-1}} - \frac{\mathsf{f}_{\alpha\beta_2}(t_1t_2^{-m_{\alpha\beta_2}},t_3t_2^{-m'_{\alpha\beta_2}},t_4 t_2^{-m''_{\alpha\beta_2}})}{1-t_2^{-1}}, \\
\mathsf{e}_{\alpha\beta_3} &:= t_3^{-1} \frac{\mathsf{f}_{\alpha\beta_3}(t_1,t_2,t_4)}{1-t_3^{-1}} - \frac{\mathsf{f}_{\alpha\beta_3}(t_1t_3^{-m_{\alpha\beta_3}},t_2t_3^{-m'_{\alpha\beta_3}},t_4 t_3^{-m''_{\alpha\beta_3}})}{1-t_3^{-1}}, \\
\mathsf{e}_{\alpha\beta_4} &:= t_4^{-1} \frac{\mathsf{f}_{\alpha\beta_4}(t_1,t_2,t_3)}{1-t_4^{-1}} - \frac{\mathsf{f}_{\alpha\beta_4}(t_1t_4^{-m_{\alpha\beta_4}},t_2t_4^{-m'_{\alpha\beta_4}},t_3 t_4^{-m''_{\alpha\beta_4}})}{1-t_4^{-1}}, \\
\mathsf{f}_{\alpha\beta_1} &:= -Z_{\alpha\beta_1}  + \frac{P_{23}}{t_2t_3} Z_{\alpha\beta_1} \overline{Z}_{\alpha\beta_1}, \\
\mathsf{f}_{\alpha\beta_2} &:= -Z_{\alpha\beta_2}  + \frac{P_{13}}{t_1t_3} Z_{\alpha\beta_2} \overline{Z}_{\alpha\beta_2}, \\
\mathsf{f}_{\alpha\beta_3} &:= -Z_{\alpha\beta_3}  + \frac{P_{12}}{t_1t_2} Z_{\alpha\beta_3} \overline{Z}_{\alpha\beta_3}, \\
\mathsf{f}_{\alpha\beta_4} &:= -Z_{\alpha\beta_4}  + \frac{P_{12}}{t_1t_2} Z_{\alpha\beta_4} \overline{Z}_{\alpha\beta_4}, 
\end{split}
\end{align}
where $P_{23} :=(1-t_2)(1-t_3)$ etc. In these formulae, there is an asymmetry with respect to the fourth leg, which we discuss below in Remark \ref{asymmetry}. Also note that these formulae become symmetric in 1,2,3 when the fourth leg is empty, i.e.~$Z_{\alpha\beta_4} = 0$. Recall that in the PT case, we impose Assumption \ref{assumption} and there is no fourth leg. We stress that, at this point, we do \emph{not} yet impose the relation $t_1t_2t_3t_4=1$.

The first observation is that $\mathsf{v}_\alpha$ and $\mathsf{e}_{\alpha\beta}$ are Laurent polynomials in the variables $t_1,t_2,t_3,t_4$. Indeed in the expression for $\mathsf{e}_{\alpha\beta_i}$ \eqref{defvef1}, both numerator and denominator vanish at $t_i = 1$, so the pole in $t_i=1$ cancels. In order to see that $\mathsf{v}_{\alpha}$ is a Laurent polynomial, one shows that it has no pole in $t_i=1$ for each $i=1,2,3,4$. Indeed, substituting
\begin{equation*} 
Z_\alpha = \frac{Z_{\alpha\beta_i}}{1-t_i} + \cdots
\end{equation*}
into \eqref{defvef1}, where $\cdots$ does not contain poles in $t_i=1$, one finds that all poles in $t_i=1$ cancel.

Now we restrict to the Calabi-Yau torus $t_1t_2t_3t_4=1$. Using definition \eqref{defvef1}, we find
\begin{align*} 
\begin{split} 
\mathsf{V}_\alpha = \mathsf{v}_\alpha + \overline{\mathsf{v}}_\alpha, \quad \mathsf{E}_{\alpha\beta} = \mathsf{e}_{\alpha\beta} + \overline{\mathsf{e}}_{\alpha\beta},
\end{split}
\end{align*}
for all $\alpha \in V(X)$ and $\alpha\beta \in E(X)$. This follows from a straight-forward calculation using $t_1t_2t_3t_4=1$ and the following two identities (and their permutations)
\begin{equation*}
P_{123} + \overline{P}_{123} = P_{1234}, \quad -\frac{P_{23}}{t_2t_3} + t_1 \frac{\overline{P}_{23}}{t_2^{-1} t_3^{-1}} = \frac{P_{234}}{t_2t_3t_4}. 
\end{equation*}

Finally we note that, after setting $t_1t_2t_3t_4=1$, $\mathsf{v}_\alpha$ and $\mathsf{e}_{\alpha\beta}$ do not have $T$-fixed part with positive coefficient. This follows from the fact that $\mathsf{V}_\alpha$ and $\mathsf{E}_{\alpha\beta}$ do not have $T$-fixed terms with positive coefficient (see comments in \cite[Def.~2.12]{CK2}). We therefore established the following:
\begin{lemma} \label{Ksqrt}
The classes $\mathsf{v}_\alpha$, $\mathsf{e}_{\alpha\beta}$ are Laurent polynomials in $t_1,t_2,t_3,t_4$. When $t_1t_2t_3t_4=1$, they satisfy the following equations
\begin{align*}
\mathsf{V}_\alpha = \mathsf{v}_\alpha + \overline{\mathsf{v}}_\alpha, \quad \mathsf{E}_{\alpha\beta} = \mathsf{e}_{\alpha\beta} + \overline{\mathsf{e}}_{\alpha\beta}. 
\end{align*}
Moreover, for $t_1t_2t_3t_4=1$, $\mathsf{v}_\alpha$, $\mathsf{e}_{\alpha\beta}$ do not have $T$-fixed part with positive coefficient.
\end{lemma}

\begin{remark}
When we want to stress the dependence on $Z_\alpha$, $Z_{\alpha\beta}$ and distinguish between the DT/PT case, we write $\mathsf{v}_{Z_\alpha}^{\DT}$, $\mathsf{v}_{Z_\alpha}^{\PT}$, $\mathsf{e}_{Z_{\alpha\beta}}^{\DT}$, $\mathsf{e}_{Z_{\alpha\beta}}^{\PT}$ for the classes introduced in this subsection.
\end{remark}

\begin{remark} \label{differentsqrt}
In the case of $I_n(\C^4,0)$, the possibility of an explicit choice of of square root of $\mathsf{V}_{\alpha}$ first appeared in \cite{NP} and was kindly explained to us by Piazzalunga. The choice of square root in \eqref{defvef1} is non-unique. For instance, in the case of $I_n(\C^4,0)$, \cite{NP} instead work with $\mathsf{v}_\alpha = Z_\alpha - P_{123} Z_\alpha \overline{Z}_\alpha$. Our choice is convenient when taking limits in Section \ref{limitssec}.
\end{remark}

\begin{remark} \label{asymmetry}
Our choice of square root \eqref{defvef1} is asymmetric in the indices $1,2,3,4$, i.e.~index $4$ is singled out. Later, when we add the insertion of Definition \ref{Nekgen} (giving $\widetilde{\mathsf{v}}_\alpha$ in \eqref{defvef2}), we want to single out the fourth direction. Putting $y=t_4$, we want $[-\widetilde{\mathsf{v}}_\alpha]$ equal to zero when $Z_\alpha$ is not scheme theoretically supported in the hyperplane $\{x_4=0\}$ and we want $[-\widetilde{\mathsf{v}}_\alpha]$ equal to the vertex of DT/PT theory of the toric 3-fold $\{x_4 = 0\} \cong \C^3$ when $Z_\alpha$ is scheme theoretically supported in $\{x_4=0\}$ (Proposition \ref{dimredprop}).
\end{remark}

\subsection{$K$-theoretic insertion} \label{sec:Kinsert}

We turn our attention to the $K$-theoretic insertion in Definition \ref{Nekgen}. Using $\mathsf{v}_{\alpha}$, $\mathsf{e}_{\alpha\beta}$ defined in the previous section, we define
\begin{align} 
\widetilde{\mathsf{v}}_\alpha &:= \mathsf{v}_\alpha - y \overline{Z}_\alpha + \sum_{i=1}^4 \frac{y \overline{Z}_{\alpha \beta_i}(t_{i'},t_{i''},t_{i'''})}{1-t_i^{-1}},  \label{defvef2}  \\
\widetilde{\mathsf{e}}_{\alpha\beta_1} &:= \mathsf{e}_{\alpha\beta_1} + \frac{t_1}{1-t_1} y \overline{Z}_{\alpha\beta_1}(t_2,t_3,t_4)  - \frac{1}{1-t_1} y \overline{Z}_{\alpha\beta_1}(t_2t_1^{-m_{\alpha\beta_1}},t_3t_1^{-m'_{\alpha\beta_1}},t_4 t_1^{-m''_{\alpha\beta_1}}), \nonumber \\
\widetilde{\mathsf{e}}_{\alpha\beta_2} &:= \mathsf{e}_{\alpha\beta_2} + \frac{t_2}{1-t_2} y \overline{Z}_{\alpha\beta_2}(t_1,t_3,t_4)  - \frac{1}{1-t_2} y \overline{Z}_{\alpha\beta_2}(t_1t_2^{-m_{\alpha\beta_2}},t_3t_2^{-m'_{\alpha\beta_2}},t_4 t_2^{-m''_{\alpha\beta_2}}), \nonumber \\
\widetilde{\mathsf{e}}_{\alpha\beta_3} &:= \mathsf{e}_{\alpha\beta_3} + \frac{t_3}{1-t_3} y \overline{Z}_{\alpha\beta_3}(t_1,t_2,t_4)  - \frac{1}{1-t_3} y \overline{Z}_{\alpha\beta_3}(t_1t_3^{-m_{\alpha\beta_3}},t_2t_3^{-m'_{\alpha\beta_3}},t_4 t_3^{-m''_{\alpha\beta_3}}), \nonumber \\
\widetilde{\mathsf{e}}_{\alpha\beta_4} &:= \mathsf{e}_{\alpha\beta_4} + \frac{t_4}{1-t_4} y \overline{Z}_{\alpha\beta_4}(t_1,t_2,t_3)  - \frac{1}{1-t_4} y \overline{Z}_{\alpha\beta_4}(t_1t_4^{-m_{\alpha\beta_4}},t_2t_4^{-m'_{\alpha\beta_4}},t_3 t_4^{-m''_{\alpha\beta_4}}). \nonumber
\end{align}
Then 
\begin{align*}
\widetilde{\mathsf{v}}_\alpha &= Z_\alpha - y \overline{Z}_\alpha - \overline{P}_{123} Z_\alpha \overline{Z}_\alpha  + \sum_{i=1}^{3} \frac{\widetilde{\mathsf{f}}_{\alpha\beta_i}(t_1,t_2,t_3,t_4)}{1-t_i} \\
&+ \frac{1}{(1-t_4)} \Big\{ - Z_{\alpha\beta_4}  - t_4 y \overline{Z}_{\alpha\beta_4}+ \overline{P}_{123} \big( \overline{Z}_\alpha Z_{\alpha\beta_4} - Z_\alpha \overline{Z}_{\alpha\beta_4} \big) + \frac{\overline{P}_{123}}{1-t_4} Z_{\alpha\beta_4} \overline{Z}_{\alpha\beta_4} \Bigg\}, \\
\widetilde{\mathsf{e}}_{\alpha\beta_1} &= t_1^{-1} \frac{\widetilde{\mathsf{f}}_{\alpha\beta_1}(t_1,t_2,t_3,t_4)}{1-t_1^{-1}} - \frac{\widetilde{\mathsf{f}}_{\alpha\beta_1}(t_1^{-1},t_2t_1^{-m_{\alpha\beta_1}},t_3t_1^{-m'_{\alpha\beta_1}},t_4 t_1^{-m''_{\alpha\beta_1}})}{1-t_1^{-1}}, \\
\widetilde{\mathsf{e}}_{\alpha\beta_2} &= t_2^{-1} \frac{\widetilde{\mathsf{f}}_{\alpha\beta_2}(t_1,t_2,t_3,t_4)}{1-t_2^{-1}} - \frac{\widetilde{\mathsf{f}}_{\alpha\beta_2}(t_1t_2^{-m_{\alpha\beta_2}},t_2^{-1},t_3t_2^{-m'_{\alpha\beta_2}},t_4 t_2^{-m''_{\alpha\beta_2}})}{1-t_2^{-1}}, \\
\widetilde{\mathsf{e}}_{\alpha\beta_3} &= t_3^{-1} \frac{\widetilde{\mathsf{f}}_{\alpha\beta_3}(t_1,t_2,t_3,t_4)}{1-t_3^{-1}} - \frac{\widetilde{\mathsf{f}}_{\alpha\beta_3}(t_1t_3^{-m_{\alpha\beta_3}},t_2t_3^{-m'_{\alpha\beta_3}},t_3^{-1},t_4 t_3^{-m''_{\alpha\beta_3}})}{1-t_3^{-1}}, \\
\widetilde{\mathsf{e}}_{\alpha\beta_4} &= t_4^{-1} \frac{\widetilde{\mathsf{f}}_{\alpha\beta_4}(t_1,t_2,t_3,t_4)}{1-t_4^{-1}} - \frac{\widetilde{\mathsf{f}}_{\alpha\beta_4}(t_1t_4^{-m_{\alpha\beta_4}},t_2t_4^{-m'_{\alpha\beta_4}},t_3 t_4^{-m''_{\alpha\beta_4}},t_4^{-1})}{1-t_4^{-1}}, \\
\widetilde{\mathsf{f}}_{\alpha\beta_1}&(t_1,t_2,t_3,t_4) := -Z_{\alpha\beta_1} - t_1y \overline{Z}_{\alpha\beta_1}  + \frac{P_{23}}{t_2t_3} Z_{\alpha\beta_1} \overline{Z}_{\alpha\beta_1}, \\
\widetilde{\mathsf{f}}_{\alpha\beta_2}&(t_1,t_2,t_3,t_4) := -Z_{\alpha\beta_2} - t_2y \overline{Z}_{\alpha\beta_2}  + \frac{P_{13}}{t_1t_3} Z_{\alpha\beta_2} \overline{Z}_{\alpha\beta_2}, \\
\widetilde{\mathsf{f}}_{\alpha\beta_3}&(t_1,t_2,t_3,t_4) := -Z_{\alpha\beta_3}  - t_3y \overline{Z}_{\alpha\beta_3} + \frac{P_{12}}{t_1t_2} Z_{\alpha\beta_3} \overline{Z}_{\alpha\beta_3}, \\
\widetilde{\mathsf{f}}_{\alpha\beta_4}&(t_1,t_2,t_3,t_4) := -Z_{\alpha\beta_4} - t_4y \overline{Z}_{\alpha\beta_4}  + \frac{P_{12}}{t_1t_2} Z_{\alpha\beta_4} \overline{Z}_{\alpha\beta_4}.
\end{align*}

\begin{lemma} 
The classes $\widetilde{\mathsf{v}}_\alpha$, $\widetilde{\mathsf{e}}_{\alpha\beta}$ are Laurent polynomials in $t_1,t_2,t_3,t_4$. 
\end{lemma}
\begin{proof}
By Lemma \ref{Ksqrt}, $\mathsf{v}_\alpha$, $\mathsf{e}_{\alpha\beta}$ are Laurent polynomials in $t_1,t_2,t_3,t_4$. Therefore, it suffices to only consider terms involving $y$, for which it is easy to see that all poles in $t_i=1$ ($i=1,2,3,4$) cancel. 
\end{proof}

\begin{remark}
When we want to stress the dependence on $Z_\alpha$, $Z_{\alpha\beta}$ and distinguish between the DT/PT case, we write $\widetilde{\mathsf{v}}_{Z_\alpha}^{\DT}$, $\widetilde{\mathsf{v}}_{Z_\alpha}^{\PT}$, $\widetilde{\mathsf{e}}_{Z_{\alpha\beta}}^{\DT}$, $\widetilde{\mathsf{e}}_{Z_{\alpha\beta}}^{\PT}$ for the classes introduced in \eqref{defvef2}.
\end{remark}

We introduce some further notations. On each chart $U_\alpha \cong \C^4$ with coordinates $(x_1^{(\alpha)},x_2^{(\alpha)},x_3^{(\alpha)},x_4^{(\alpha)})$ the $(\C^*)^4$-action is given by 
\begin{equation*}
t \cdot x_i^{(\alpha)} = \chi_i^{(\alpha)}(t) \, x_i^{(\alpha)},  \quad \forall \, \, t=(t_1,t_2,t_3,t_4) \in (\C^*)^4
\end{equation*}
for some characters $\chi_i^{(\alpha)} : (\C^*)^4 \rightarrow \C^*$ with $i=1,2,3,4$. 
We recall that both fixed loci $I^T$, $P^T$ consists of finitely many fixed points, giving rise to local data 
$$Z = \big(\{Z_\alpha\}_{\alpha \in V(X)}, \{Z_{\alpha\beta}\}_{\alpha\beta \in E(X)}\big).$$ 
In DT case, $Z_{\alpha}$ are point- or curve-like solid partitions and $Z_{\alpha\beta}$ are finite plane partitions. In the PT case, $Z_{\alpha}$ are box configurations (see \eqref{PTZalpha}) and $Z_{\alpha\beta}$ are finite plane partitions. Recall that  in the stable pairs case we impose 
Assumption \ref{assumption} from the introduction. Finally, we introduce the following notation. For any $t_1^{w_1}t_2^{w_2}t_3^{w_3}t_4^{w_4} y^{a} \in K_0^{T\times \C^*}(\mathrm{pt})$, we set
\begin{equation} \label{bracket}
[t_1^{w_1}t_2^{w_2}t_3^{w_3}t_4^{w_4} y^{a}] := t_1^{\frac{w_1}{2}}t_2^{\frac{w_2}{2}}t_3^{\frac{w_3}{2}}t_4^{\frac{w_4}{2}} y^{\frac{a}{2}} - t_1^{-\frac{w_1}{2}}t_2^{-\frac{w_2}{2}}t_3^{-\frac{w_3}{2}}t_4^{-\frac{w_4}{2}} y^{-\frac{a}{2}}
\end{equation}
and we extend this definition to $K_0^{T \times \C^*}(\mathrm{pt})$ by setting
$$
\Big[\sum_a \tau^a - \sum_b \tau^b\Big] := \frac{\prod_a [\tau^a]}{\prod_b [\tau^b]}, 
$$
where we use multi-index notation $\tau := (t_1,t_2,t_3,t_4,y)$ and this expression is only defined when no index $b$ equals $0$.

\begin{theorem} \label{vertexthm}
Let $X$ be a toric Calabi-Yau 4-fold, $\beta \in H_2(X,\Z)$ and $n \in \Z$. Let $L$ be a $T$-equivariant line bundle on $X$ and denote the character of $L|_{U_\alpha}$ by
$\gamma^{(\alpha)}(t) \in K_0^T(U_\alpha)$.
Then
\begin{align*}
I_{n,\beta}(L,y) = \sum_{Z\in I_{n}(X,\beta)^T} (-1)^{o(\mathcal{L})|_Z} \bigg(\prod_{\alpha \in V(X)} [-\widetilde{\mathsf{v}}^{\DT}_{Z_\alpha}]  \bigg) \bigg(\prod_{\alpha\beta \in E(X)} [-\widetilde{\mathsf{e}}^{\DT}_{Z_{\alpha\beta}}] \bigg), \\
P_{n,\beta}(L,y) = \sum_{Z\in P_{n}(X,\beta)^T} (-1)^{o(\mathcal{L})|_Z} \bigg( \prod_{\alpha \in V(X)} [-\widetilde{\mathsf{v}}^{\PT}_{Z_\alpha}] \bigg)  \bigg( \prod_{\alpha\beta \in E(X)} [-\widetilde{\mathsf{e}}^{\PT}_{Z_{\alpha\beta}}] \bigg),
\end{align*}
where the sums are over all $T$-fixed points $Z =\big(\{Z_\alpha\}_{\alpha \in V(X)}, \{Z_{\alpha\beta}\}_{\alpha\beta \in E(X)}\big)$ and  $\widetilde{\mathsf{v}}_{Z_\alpha}, \widetilde{\mathsf{e}}_{Z_{\alpha\beta}}$ are evaluated at $$(t_1,t_2,t_3,t_4,y) = (\chi_1^{(\alpha)}(t),\chi_2^{(\alpha)}(t),\chi_3^{(\alpha)}(t),\chi_4^{(\alpha)}(t),\overline{\gamma}^{(\alpha)}(t) \cdot y).$$
\end{theorem}

\begin{proof}
We discuss the DT case; the PT case is similar.  
We suppose $L = \O_X$ and discuss the general case afterwards. For any $Z \in I^T$, 
we have  
\begin{align*}
I_{n,\beta}(\O_X, y)=\sum_{Z \in I^{T}} (-1)^{o(\mathcal{L})|_Z} \frac{\ch\left(\sqrt{K_I^{\vir}|_{Z}}^{\frac{1}{2}}\right)}{\ch\left(\Lambda^{\mdot} \sqrt{T^{\vir}|_{Z}}^{\vee}\right)}   \frac{\ch(\Lambda^{\mdot} (\O_X^{[n]}|_{Z} \otimes y^{-1}))}{\ch((\det(\O_X^{[n]} |_{Z} \otimes y^{-1}))^{\frac{1}{2}})}.
\end{align*}
By Lemma \ref{Ksqrt}, 
\begin{equation} \label{eqn1}
\sqrt{T_I^{\vir}|_{Z}} = \sum_{\alpha \in V(X)} \mathsf{v}_{Z_\alpha} + \sum_{\alpha\beta \in E(X)} \mathsf{e}_{Z_{\alpha\beta}}. 
\end{equation}
At any fixed point $Z\in I^T$, we have
\begin{align*}
\frac{\Lambda^{\mdot} (\O_X^{[n]} \otimes y^{-1})} {(\det(     \O_X^{[n]} \otimes y^{-1} ))^{\frac{1}{2}}  }  \Bigg|_{Z} = \frac{\Lambda^{\mdot} R\Gamma(X, \O_Z \otimes y^{-1})} {(\det(     R\Gamma(X, \O_{Z} \otimes y^{-1}))^{\frac{1}{2}})  }, 
\end{align*}
where $\O_X^{[n]}$ was defined in \eqref{taut}. Calculation by \v{C}ech cohomology gives 
\begin{equation} \label{eqn2}
\tr_{R\Gamma(X, \O_Z \otimes y^{-1})} = \sum_{\alpha \in V(X)} \tr_{\Gamma(U_\alpha, \O_{Z_\alpha})} \otimes y^{-1} - \sum_{\alpha\beta \in E(X)} \tr_{\Gamma(U_{\alpha\beta}, \O_{Z_{\alpha\beta}})} \otimes y^{-1}.
\end{equation}
As $T$-representations, $\tr_{\Gamma(U_\alpha, \O_{Z_\alpha})} = Z_\alpha$, where $Z_{\alpha}$ was defined in \eqref{Zalpha}. Suppose $\PP^1 \cong L_{\alpha\beta} = \{x_2^{(\alpha)}=x_3^{(\alpha)}=x_4^{(\alpha)}=0\}$, i.e.~leg $Z_{\alpha\beta}$ is along the $x_1^{(\alpha)}$-axis, then
$$
\tr_{\Gamma(U_{\alpha\beta}, \O_{Z_{\alpha\beta}})} = \delta(\chi^{(\alpha)}_1(t)) \, Z_{\alpha\beta},
$$
where $\chi^{(\alpha)}_1(t)$ denotes the character corresponding to the $(\C^*)^4$-action on the first coordinate of $U_{\alpha\beta} \cong \C^* \times \C^3$ and $\delta(t)$ was defined in \eqref{defdelta}.
Next, we use the following essential identity
\begin{equation*}
\ch(\Lambda^\mdot \mathcal{L}^*) = \frac{e(\mathcal{L})}{\td(\mathcal{L})} = 1 - e^{-c_1(\mathcal{L})},
\end{equation*}
for any $T$-equivariant line bundle $\mathcal{L}$. Hence
\begin{equation} \label{essential}
\frac{\ch(\Lambda^\mdot \mathcal{L}^*)}{ \ch((\det \mathcal{L}^*)^{\frac{1}{2}})} = \frac{1 - e^{-c_1(\mathcal{L})}}{ e^{-\frac{1}{2} c_1(\mathcal{L}) }} = e^{\frac{c_1(\mathcal{L})}{2}} - e^{-\frac{c_1(\mathcal{L})}{2}}.
\end{equation}
Furthermore, we recall the following relations in the $K$-group
\begin{align}
\begin{split} \label{eqnK}
\Lambda^\mdot (E \oplus F) &= \Lambda^\mdot E \otimes \Lambda^\mdot F, \\
\det(E \oplus F) &= \det E \otimes \det F.
\end{split}
\end{align}

From \eqref{eqn1}, \eqref{eqn2}, \eqref{eqnK}, and the fact that $\ch(\cdot)$ is a ring homomorphism, we at once deduce
\begin{align}
\begin{split} \label{eqn3}
I_{n,\beta}(\O_X, y) = \sum_{Z \in I^{T}} (-1)^{o(\mathcal{L})|_Z} &\prod_{\alpha \in V(X)} \frac{\ch((\det \overline{\mathsf{v}}_{Z_\alpha})^{\frac{1}{2}} )}{\ch(\Lambda^{\mdot} \overline{\mathsf{v}}_{Z_\alpha} )} \frac{\ch(\Lambda^{\mdot} ( \O_{Z_\alpha} \otimes y^{-1}))}{\ch((\det(\O_{Z_\alpha} \otimes y^{-1}))^{\frac{1}{2}})}  \\
&\prod_{\alpha\beta \in E(X)} \frac{\ch((\det \overline{\mathsf{e}}_{Z_{\alpha\beta}})^{\frac{1}{2}} )}{\ch(\Lambda^{\mdot} \overline{\mathsf{e}}_{Z_{\alpha\beta}} )} \frac{\ch(\Lambda^{\mdot} ( \O_{Z_{\alpha\beta}} \otimes y^{-1}))}{\ch((\det(\O_{Z_{\alpha\beta}} \otimes y^{-1}))^{\frac{1}{2}})},
\end{split}
\end{align}
where $\mathsf{v}_{Z_\alpha}, \mathsf{e}_{Z_{\alpha\beta}}$ were defined in \eqref{defvef1}. Using multi-index notation for $t = (t_1,t_2,t_3,t_4)$, where $t_1t_2t_3t_4=1$, we write
\begin{align*}
Z_{\alpha} &= \sum_{u_{Z_\alpha}} t^{u_{Z_\alpha}}, \quad \mathsf{v}_{Z_\alpha} = \sum_{v_{Z_\alpha}} t^{v_{Z_\alpha}} - \sum_{w_{Z_\alpha}} t^{w_{Z_\alpha}}, \\
Z_{\alpha\beta} &= \sum_{u_{Z_{\alpha\beta}}} t^{u_{Z_{\alpha\beta}}}, \quad \mathsf{e}_{Z_{\alpha\beta}} = \sum_{v_{Z_{\alpha\beta}}} t^{v_{Z_{\alpha\beta}}} - \sum_{w_{Z_{\alpha\beta}}} t^{w_{Z_{\alpha\beta}}}.
\end{align*}
Note that $ \mathsf{v}_{Z_\alpha}$, $\mathsf{e}_{Z_{\alpha\beta}}$, $Z_{\alpha\beta}$ are all Laurent polynomials (Lemma \ref{Ksqrt}). However, $Z_\alpha$ is in general a Laurent series (we will redistribute its term shortly). Combining \eqref{essential} and \eqref{eqn3}, we find
\begin{align*}
I_{n,\beta}(\O_X, y) = &\sum_{Z \in I^{T}} (-1)^{o(\mathcal{L})|_Z} \prod_{\alpha \in V(X)} \prod_{u_{Z_\alpha}, v_{Z_\alpha}, w_{Z_\alpha}} \Bigg( \frac{t^{\frac{w_{Z_\alpha}}{2}}  - t^{-\frac{w_{Z_\alpha}}{2}}}{t^{\frac{v_{Z_\alpha}}{2}}  - t^{-\frac{v_{Z_\alpha}}{2}}} \Bigg) (t^{-\frac{u_{Z_\alpha}}{2}} y^{\frac{1}{2}}  - t^{\frac{u_{Z_\alpha}}{2}} y^{-\frac{1}{2}})  \\
&\prod_{\alpha\beta \in E(X)} \prod_{u_{Z_{\alpha\beta}}, v_{Z_{\alpha\beta}}, w_{Z_{\alpha\beta}}}  \Bigg( \frac{t^{\frac{w_{Z_{\alpha}\beta}}{2}}  - t^{-\frac{w_{Z_{\alpha\beta}}}{2}}}{t^{\frac{v_{Z_{\alpha\beta}}}{2}}  - t^{-\frac{v_{Z_{\alpha\beta}}}{2}}} \Bigg) (t^{-\frac{u_{Z_{\alpha\beta}}}{2}} y^{\frac{1}{2}} - t^{\frac{u_{Z_{\alpha\beta}}}{2}} y^{-\frac{1}{2}}) \\
= &\sum_{Z \in I^{T}}(-1)^{o(\mathcal{L})|_Z} \prod_{\alpha \in V(X)} [-\mathsf{v}_\alpha] \cdot \prod_{u_{Z_\alpha}}   (t^{-\frac{u_{Z_\alpha}}{2}} y^{\frac{1}{2}}  - t^{\frac{u_{Z_\alpha}}{2}} y^{-\frac{1}{2}})  \\
&\prod_{\alpha\beta \in E(X)}  [-\mathsf{e}_{\alpha\beta}] \cdot \prod_{u_{Z_{\alpha\beta}}} (t^{-\frac{u_{Z_{\alpha\beta}}}{2}} y^{\frac{1}{2}} - t^{\frac{u_{Z_{\alpha\beta}}}{2}} y^{-\frac{1}{2}}),
\end{align*}
where, for each $\alpha \in V(X)$ and $\alpha\beta \in E(X)$, the corresponding terms in the product are evaluated in
$$
t = (t_1,t_2,t_3,t_4) = (\chi_1^{(\alpha)}(t),\chi_2^{(\alpha)}(t),\chi_3^{(\alpha)}(t),\chi_4^{(\alpha)}(t)).
$$
The terms in the products over $u_{Z_\alpha}, u_{Z_{\alpha\beta}}$ can be ``absorbed'' into $\mathsf{v}_\alpha$, $\mathsf{e}_{\alpha\beta}$. This can be achieved by distributing 
$$
-\sum_{\alpha \in V(X)}  \overline{Z}_\alpha \otimes y +  \sum_{\alpha\beta \in E(X)}  \delta(\chi^{(\alpha)}_1(t)) \, \overline{Z}_{\alpha\beta} \otimes y 
$$
over $\mathsf{v}_\alpha$, $\mathsf{e}_{\alpha\beta}$ as in \eqref{defvef2}. We conclude
\begin{equation*}
I_{n,\beta}(\O_X, y) =\sum_{Z \in I^{T}} (-1)^{o(\mathcal{L})|_Z} \bigg( \prod_{\alpha \in V(X)} [-\widetilde{\mathsf{v}}_\alpha] \bigg)  \bigg( \prod_{\alpha\beta \in E(X)}  [-\widetilde{\mathsf{e}}_{\alpha\beta}]\bigg). 
\end{equation*}
This finishes the case $L  = \O_X$. Replacing $y$ by
$$
\overline{\gamma}^{(\alpha)}(t) \cdot y, \quad \overline{\gamma}^{(\alpha\beta)}(t) \cdot y
$$
establishes the general case. 
\end{proof}

\begin{remark}
In the case $I_n(\C^4,0)$, which is discussed in \cite{NP}, our definition of $\widetilde{\mathsf{v}}_{Z_\alpha}$ \eqref{defvef2} differs slightly from loc.~cit., who take $(1-y^{-1})Z_\alpha - P_{123} Z_\alpha \overline{Z}_\alpha$. The difference of the second term was discussed in Remark \ref{differentsqrt}. The difference of the $y$-term is explained as follows. For $I_n(\C^4,0)$, Nekrasov-Piazzalunga consider the invariant defined in Definition \ref{Nekgen} but with $L^{[n]} \otimes y^{-1}$ replaced by $(L^{[n]})^{\vee} \otimes y$ (and $L = \O_X$). Note the following two identities
\begin{align*}
\Lambda^\mdot E^\vee &= (-1)^{\rk(E)} \Lambda^\mdot E \otimes \det(E)^*, \\
\det(E^{\vee}) &= \det(E)^*,
\end{align*}
for any $E \in K_0^{T \times \C^*}(\mathrm{pt})$. This shows at once that our definition differs from loc.~cit.~by an overall factor $(-1)^n$. In the vertex formalism, it results in replacing $y \overline{Z}_\alpha$ by $y^{-1} Z_\alpha$ in \eqref{defvef2}.
\end{remark}

We end with an observation about the powers of the equivariant parameters. The expressions $[-\widetilde{\mathsf{v}}_{\alpha}]$, $[-\widetilde{\mathsf{e}}_{\alpha\beta}]$ a priori involve \emph{half-integer} powers of $t_1,t_2,t_3,t_4$ (formal square roots). In fact, taking a single leg of multiplicity one with weights $m_{\alpha\beta} = 0$, $m'_{\alpha\beta} = -1$, $m''_{\alpha\beta} = -1$ already shows that non-integer powers indeed occur in the edge. Nonetheless, for the vertex we have the following:
\begin{proposition} \label{integerpowers}
We have 
$$
[-\widetilde{\mathsf{v}}_{\alpha}] \in \Q(t_1,t_2,t_3,t_4,y^{\frac{1}{2}}) / (t_1t_2t_3t_4 - 1).
$$
\end{proposition}
\begin{proof}
We first consider the case that $Z_\alpha$ satisfies $Z_{\alpha\beta_1}=Z_{\alpha\beta_2}=Z_{\alpha\beta_3}=Z_{\alpha\beta_4} = 0$. As before, we will use multi-index notation for $t = (t_1,t_2,t_3,t_4)$. A monomial $ \pm t^v$ in $\widetilde{\mathsf{v}}_{\alpha}$ contributes as follows to $[-\widetilde{\mathsf{v}}_{\alpha}]$:
$$
[\mp t^v] = (t^{\frac{v}{2}} - t^{-\frac{v}{2}})^{\mp 1} = \Big( \frac{1- t^{-v}}{t^{-\frac{v}{2}}}  \Big)^{\mp 1}.
$$
Hence, non-integer powers can only come from $t^{ \mp \frac{v}{2}} = (\det ( \mp t^{v}))^{\frac{1}{2}}$. Therefore, it suffices to calculate $(\det(\cdot))^{\frac{1}{2}}$ of
$$
\widetilde{\mathsf{v}}_{\alpha} = Z_\alpha - y \overline{Z}_\alpha - \overline{P}_{123} Z_\alpha \overline{Z}_\alpha.
$$
Writing $Z_\alpha = \sum_{u} t^u$, we find
$$
(\det \widetilde{\mathsf{v}}_\alpha)^{\frac{1}{2}} =  \prod_u \frac{t^{\frac{u}{2}}}{y^{\frac{1}{2}} t^{-\frac{u}{2}} } = \prod_u y^{-\frac{1}{2}}t^u,
$$
where we used that $\det (\overline{P}_{123} Z_\alpha \overline{Z}_\alpha) = 1$.

For the general case, write $Z_\alpha = \sum_{i=1}^4 \frac{Z_{\alpha\beta_i}}{1-t_i} + W$, where $W$ is a Laurent polynomial. Next, substitute this expression for $Z_\alpha$ into definition \eqref{defvef2} of $\widetilde{\mathsf{v}}_{\alpha}$ and cancel all poles. Up to this point in the calculation, the relation $t_1t_2t_3t_4=1$ is \emph{not} imposed. Then, similar to the calculation above, taking $(\det(\cdot))^{\frac{1}{2}}$ of the resulting Laurent polynomial gives only integer powers of $t_i$. 
\end{proof}

\subsection{$K$-theoretic 4-fold vertex}

Let $\lambda, \mu, \nu, \rho$ be plane partitions of finite size. This determines a $T$-fixed Cohen-Macaulay curve  $C \subseteq \C^4$ with solid partition defined by \eqref{CMsolid}. We denote this solid partition by $\pi_{\CM}(\lambda,\mu,\nu,\rho)$. Consider the following:
\begin{itemize}
\item All $T$-fixed closed subschemes $Z \subseteq \C^4$ with underlying maximal Cohen-Macaulay subcurve $C$. These correspond to solid partitions $\pi$ with asymptotic plane partitions $\lambda,\mu,\nu,\rho$ in directions $1,2,3,4$. We denote the collection of such solid partitions by $\Pi^{\DT}(\lambda,\mu,\nu,\rho)$. Any $\pi \in \Pi^{\DT}(\lambda,\mu,\nu,\rho)$ determines a character $Z_\pi$ defined by the RHS of \eqref{Zalpha} and hence, by \eqref{defvef2}, a Laurent polynomial
$$
\widetilde{\mathsf{v}}_{\pi}^{\DT} \in \Q[t_1^{\pm 1},t_2^{\pm 1},t_3^{\pm 1},t_4^{\pm 1},y] / (t_1t_2t_3t_4-1).
$$ 
\item Assume at most two of $\lambda, \mu, \nu, \rho$ are non-empty. Consider all $T$-fixed stable pairs $(F,s)$ on $X$ with underlying Cohen-Macaulay support curve $C$. These correspond to box configurations as described in Section \ref{fixlocus}. We denote the collection of these box configurations by $\Pi^{\PT}(\lambda,\mu,\nu,\rho)$. Any $B \in \Pi^{\PT}(\lambda,\mu,\nu,\rho)$ determines a character $Z_B$ defined by the RHS of \eqref{PTZalpha}, where the Cohen-Macaulay part is given by \eqref{Zalpha} with solid partition $\pi_{\CM}(\lambda,\mu,\nu,\rho)$. By \eqref{defvef2}, this determines a Laurent polynomial
$$
\widetilde{\mathsf{v}}_{B}^{\PT} \in \Q[t_1^{\pm 1},t_2^{\pm 1},t_3^{\pm 1},t_4^{\pm 1},y] / (t_1t_2t_3t_4-1).
$$ 
\end{itemize}

\begin{definition} \label{vertexdef}
Let $\lambda, \mu, \nu, \rho$ be plane partitions of finite size. Define the DT 4-fold vertex by
\begin{align*}
\mathsf{V}_{\lambda\mu\nu\rho}^{\DT}(t,y,q)_{o(\L)}& := \sum_{\pi \in \Pi^{\DT}(\lambda, \mu, \nu, \rho)} (-1)^{o(\L)|_\pi} [-\widetilde{\mathsf{v}}^{\DT}_\pi] \, q^{|\pi|} \\ & \in \Q(t_1,t_2,t_3,t_4,y^{\frac{1}{2}}) / (t_1t_2t_3t_4-1)(\!(q)\!),
\end{align*}
where $o(\L)|_\pi = 0,1$ denotes a choice of sign for each $\pi$, $[\cdot]$ was defined in \eqref{bracket}, $|\pi|$ denotes renormalized volume (Definition \ref{solid}), and RHS is well-defined by Lemma \ref{Ksqrt}. Note that the powers of $t_i$ are integer by Proposition \ref{integerpowers}. 

Next, suppose at most two of $\lambda, \mu, \nu, \rho$ are non-empty. Define the PT 4-fold vertex by
\begin{align*}
\mathsf{V}_{\lambda\mu\nu\rho}^{\PT}(t,y,q)_{o(\L)} &:= \sum_{B \in \Pi^{\PT}(\lambda, \mu, \nu, \rho)} (-1)^{o(\L)|_B} [-\widetilde{\mathsf{v}}^{\PT}_B] \, q^{|B| + |\pi_{\mathrm{CM}}(\lambda, \mu, \nu, \rho)|} \\ &
\in \Q(t_1,t_2,t_3,t_4,y^{\frac{1}{2}}) / (t_1t_2t_3t_4-1)(\!(q)\!),
\end{align*}
where $o(\L)|_B = 0,1$ denotes a choice of sign for each $B$, $|B|$ denotes the total number of boxes in the box configuration, and $|\pi_{\mathrm{CM}}(\lambda, \mu, \nu, \rho)|$ denotes renormalized volume. We often omit $o(\L)$ from the notation. 
\end{definition}

Similarly, to any finite plane partition $\lambda$, we associate a character $Z_\lambda$ defined by RHS of \eqref{Zalphabeta}. We then define edge terms
\begin{align*}
\mathsf{E}_\lambda^{\DT}(t,y) = \mathsf{E}_\lambda^{\PT}(t,y) := (-1)^{o(\L)|_{\lambda}} [-\widetilde{\mathsf{e}}_{Z_\lambda}] \in 
\Q(t_1^{\frac{1}{2}},t_2^{\frac{1}{2}},t_3^{\frac{1}{2}},t_4^{\frac{1}{2}},y^{\frac{1}{2}}) / (t_1t_2t_3t_4-1),
\end{align*}
where $\widetilde{\mathsf{e}}_{Z_\lambda}$ was defined in \eqref{defvef2}.

The vertex formalism reduces the calculation of $I_{n,\beta}(L,y), P_{n,\beta}(L,y)$ for any toric Calabi-Yau 4-fold $X$, $\beta \in H_2(X,\Z)$, and $n \in \Z$ to a combinatorial expression involving $\mathsf{V}_{\lambda\mu\nu\rho}$ and $\mathsf{E}_{\lambda}$. We illustrate this in a sufficiently general example. Let $X$ be the total space of $\O_{\PP^2}(-1) \oplus \O_{\PP^2}(-2)$. Let $\beta = d\,[\PP^1]$, where $\PP^1$ lies in the zero section $\PP^2 \subseteq X$, and let $L$ be a $T$-equivariant line bundle on $X$. Denote the characters of $L|_{U_\alpha}$, $L|_{U_{\alpha\beta}}$ by
$\gamma^{(\alpha)}(t) \in K_0^T(U_\alpha)$, $\gamma^{(\alpha\beta)}(t) \in K_0^T(U_{\alpha\beta})$ for all $\alpha=1,2,3$ and all $\alpha\beta$.
Then Lemma \ref{chi} and Theorem \ref{vertexthm} imply
\begin{align*}
&\sum_n I_{n,\beta}(L,y) \, q^n =\sum_{\lambda,\mu,\nu \atop |\lambda|+|\mu|+|\nu| = d} q^{f_{1,-1,-2}(\lambda)+f_{1,-1,-2}(\mu)+f_{1,-1,-2}(\nu)} \\
&\cdot \mathsf{E}^{\DT}_{\lambda}|_{(t_1,t_2,t_3,t_4,\overline{\gamma}^{(13)}(t_2,t_3,t_4)y)} \mathsf{V}_{\lambda\mu\varnothing\varnothing}^{\DT}|_{(t_1,t_2,t_3,t_4,\overline{\gamma}^{(1)}(t_1,t_2,t_3,t_4) y)} \mathsf{E}^{\DT}_{\mu} |_{(t_2,t_1,t_3,t_4,\overline{\gamma}^{(12)}(t_1,t_3,t_4) y)}  \\
&\cdot \mathsf{V}^{\DT}_{\mu\nu\varnothing\varnothing} |_{(t_2^{-1},t_1 t_2^{-1},t_3 t_2,t_4t_2^2, \overline{\gamma}^{(2)}(t_2^{-1},t_1 t_2^{-1},t_3 t_2,t_4t_2^2) y) } \mathsf{E}^{\DT}_{\nu}|_{(t_1 t_2^{-1}, t_2^{-1},t_3 t_2,t_4 t_2^2, \overline{\gamma}^{(23)}(t_2^{-1},t_3 t_2,t_4t_2^2) y)} \\
&\cdot \mathsf{V}^{\DT}_{\nu\lambda\varnothing\varnothing}|_{(t_2t_1^{-1},t_1^{-1},t_3 t_1,t_4 t_1^2,\overline{\gamma}^{(3)}(t_2t_1^{-1},t_1^{-1},t_3 t_1,t_4 t_1^2) y)},
\end{align*}
where the sum is over all finite plane partitions $\lambda, \mu, \nu$ satisfying $|\lambda|+|\mu|+|\nu| = d$. Here the choice of signs for the invariants $ I_{n,\beta}(L,y)$ is determined by the choice of signs in each vertex and edge term. Replacing $\DT$ by $\PT$, the same expression holds for the generating function of $P_{n,\beta}(L,y)$.

We conjecture that the DT/PT 4-fold vertex satisfy Conjecture \ref{K-conj intro}. As above, for finite plane partitions $\lambda, \mu, \nu, \rho$, we denote by $\pi_{\CM}(\lambda,\mu,\nu,\rho)$ the curve-like solid partition corresponding to the Cohen-Macaulay curve with ``asymptotics'' $\lambda, \mu, \nu, \rho$. We normalize the DT/PT 4-fold vertex so they start with $q^0$ (whose coefficient is in general not equal to 1). 
This is achieved by multiplying by $q^{-|\pi_{\mathrm{CM}}(\lambda,\mu,\nu,\rho)|}$.

Using the vertex formalism, we verified the following cases:
\begin{proposition}\label{verif} 
There are choices of signs such that 
$$
q^{-|\pi_{\CM}(\lambda,\mu,\nu,\rho)|} \, \mathsf{V}_{\lambda\mu\nu\rho}^{\DT}(t,y,q) = q^{-|\pi_{\CM}(\lambda,\mu,\nu,\rho)|} \, \mathsf{V}_{\lambda\mu\nu\rho}^{\PT}(t,y,q) \, \mathsf{V}_{\varnothing\varnothing\varnothing\varnothing}^{\DT}(t,y,q) \mod q^N
$$
in the following cases:
\begin{itemize}
\item for any $|\lambda| + |\mu| + |\nu| + |\rho| \leqslant 1$ and $N=4$,
\item for any $|\lambda| + |\mu| + |\nu| + |\rho| \leqslant 2$ and $N=4$,
\item for any $|\lambda| + |\mu| + |\nu| + |\rho| \leqslant 3$ and $N=3$,
\item for any $|\lambda| + |\mu| + |\nu| + |\rho| \leqslant 4$ and $N=3$.
\end{itemize}
In each of these cases, the uniqueness statement of Conjecture \ref{K-conj intro} holds.
\end{proposition}
\begin{remark}[Sign rule] \label{sign expec}
Let $X$ be a toric Calabi-Yau 4-fold and $\beta \in H_2(X,\Z)$. Let $Z =\{\{Z_\alpha\}_{\alpha \in V(X)}, \{Z_{\alpha\beta}\}_{\alpha\beta \in E(X)}\}$ be an element of the fixed locus $\bigcup_n I_n(X,\beta)^T$. Consider charts $U_\alpha \cong \C^4$ and $U_{\alpha\beta} \cong \C^* \times \C^3$. Suppose that $Z_\alpha, Z_{\alpha\beta}$ are set theoretically (but not necessarily scheme theoretically!) contained in $\{x_4=0\}$.
Denote by $\pi$ the solid partition corresponding to $Z_\alpha$, then we define
$$
\sigma_{\DT}(Z_\alpha) := (-1)^{|\pi|} \prod_{w=(a,a,a,d) \in \pi \atop a < d} (-1)^{1-\# \{\mathrm{legs \ containing  \ } w\}},
$$
where $(a,a,a,d) \in \pi$ means $1 \leqslant d \leqslant \pi_{aaa}$. When $Z_\alpha$ is 0-dimensional, this reduces to 
$
\sigma_{\DT}(Z_\alpha) = (-1)^{|\pi| + \# \{ (a,a,a,d) \in \pi \, : \,  a < d\} },
$
which coincides with the sign discovered by Nekrasov-Piazzalunga \cite[(2.60)]{NP}.\footnote{In loc.~cit.~the sign is $ (-1)^{|\pi| + \# \{(a,a,a,d) \in \pi \, : \, a \leqslant d\} }$. The difference is explained by the fact that they use a different choice of square root as discussed in Remark \ref{differentsqrt}. See also \cite{Mon} for further discussions on the sign rules.} 
Then $\sigma_{\DT}(Z_\alpha)$ produces the correct unique signs (up to overall sign) in our verifications of Conjecture \ref{K-conj intro} as listed in Proposition \ref{verif}. A similar formula appears to hold on the stable pair side.\footnote{Then $Z_\alpha$ corresponds to a Cohen-Macaulay support curve described by a solid partition $\pi_{\CM}$ together with a box configuration $B \subseteq \Z^4$ (Section \ref{fixlocus}). The sign $\sigma_{\PT}(Z_\alpha) := (-1)^{|\pi_{\CM}| + |B| + \# \{ (a,a,a,d) \in B \, : \,  a < d\}} \prod_{w=(a,a,a,d) \in \pi_{\CM} \atop a < d} (-1)^{1-\# \{\mathrm{legs \ containing  \ } w\}}$ works in all our calculations.} 
Denote by $\lambda$ the plane partition corresponding to $Z_{\alpha\beta}$, then we define
\begin{align*}
\sigma_{\mathrm{edge}}(Z_{\alpha\beta}):=(-1)^{f(\alpha,\beta)+|\lambda|\cdot {m}''_{\alpha\beta}}\prod_{(a,a,d)\in \lambda \atop a<d}(-1).
\end{align*}
This produces the unique signs in our verifications of Conjecture \ref{localrescon} as listed in Proposition \ref{verif localrescon} (and it is also consistent with the signs required in the calculations in \cite{CKM}). 
\end{remark}

We now show that Conjecture \ref{K-conj intro} implies Theorem \ref{globaltoricKDTPT} (global $K$-theoretic DT/PT correspondence).
\begin{proof}[Proof of Theorem \ref{globaltoricKDTPT}]
For ease of notation, we consider the case where $X=\mathrm{Tot}_{\mathbb{P}^2}(\O(-1) \oplus \O(-2))$ and $\beta=d\,[\PP^1]$. The general case follows similarly. Conjecture \ref{K-conj intro} implies that there exist choices of signs such that
\begin{align*}
&\sum_n I_{n,\beta}(L,y) \, q^n =\sum_{\lambda,\mu,\nu \atop |\lambda|+|\mu|+|\nu| = d} q^{f_{1,-1,-2}(\lambda)+f_{1,-1,-2}(\mu)+f_{1,-1,-2}(\nu)} \\
&\cdot \mathsf{E}^{\DT}_{\lambda}|_{(t_1,t_2,t_3,t_4,\overline{\gamma}^{(13)}(t_2,t_3,t_4) y)} \mathsf{V}_{\lambda\mu\varnothing\varnothing}^{\DT}|_{(t_1,t_2,t_3,t_4,\overline{\gamma}^{(1)}(t_1,t_2,t_3,t_4) y)} \mathsf{E}^{\DT}_{\mu} |_{(t_2,t_1,t_3,t_4,\overline{\gamma}^{(12)}(t_1,t_3,t_4) y)}  \\
&\cdot \mathsf{V}^{\DT}_{\mu\nu\varnothing\varnothing} |_{(t_2^{-1},t_1 t_2^{-1},t_3 t_2,t_4t_2^2, \overline{\gamma}^{(2)}(t_2^{-1},t_1 t_2^{-1},t_3 t_2,t_4t_2^2) y) } \mathsf{E}^{\DT}_{\nu}|_{(t_1 t_2^{-1}, t_2^{-1},t_3 t_2,t_4 t_2^2, \overline{\gamma}^{(23)}(t_2^{-1},t_3 t_2,t_4t_2^2) y)} \\
&\cdot \mathsf{V}^{\DT}_{\nu\lambda\varnothing\varnothing}|_{(t_2t_1^{-1},t_1^{-1},t_3 t_1,t_4 t_1^2,\overline{\gamma}^{(3)}(t_2t_1^{-1},t_1^{-1},t_3 t_1,t_4 t_1^2) y)} \\
&=\mathsf{V}_{\varnothing\varnothing\varnothing\varnothing}^{\DT}\Big|_{(t_1,t_2,t_3,t_4,\overline{\gamma}^{(1)}(t_1,t_2,t_3,t_4) y)}  \cdot \mathsf{V}^{\DT}_{\varnothing\varnothing\varnothing\varnothing}|_{(t_2^{-1},t_1 t_2^{-1},t_3 t_2,t_4t_2^2, \overline{\gamma}^{(2)}(t_2^{-1},t_1 t_2^{-1},t_3 t_2,t_4t_2^2) y) } \\
&\cdot \mathsf{V}^{\DT}_{\varnothing\varnothing\varnothing\varnothing}|_{(t_2t_1^{-1},t_1^{-1},t_3 t_1,t_4 t_1^2,\overline{\gamma}^{(3)}(t_2t_1^{-1},t_1^{-1},t_3 t_1,t_4 t_1^2) y)}  \sum_{\lambda,\mu,\nu \atop |\lambda|+|\mu|+|\nu| = d} q^{f_{1,-1,-2}(\lambda)+f_{1,-1,-2}(\mu)+f_{1,-1,-2}(\nu)} \\
&\cdot \mathsf{E}^{\PT}_{\lambda}|_{(t_1,t_2,t_3,t_4,\overline{\gamma}^{(13)}(t_2,t_3,t_4) y)} \mathsf{V}_{\lambda\mu\varnothing\varnothing}^{\PT}|_{(t_1,t_2,t_3,t_4,\overline{\gamma}^{(1)}(t_1,t_2,t_3,t_4) y)} \mathsf{E}^{\PT}_{\mu} |_{(t_2,t_1,t_3,t_4,\overline{\gamma}^{(12)}(t_1,t_3,t_4) y)}  \\
&\cdot \mathsf{V}^{\PT}_{\mu\nu\varnothing\varnothing} |_{(t_2^{-1},t_1 t_2^{-1},t_3 t_2,t_4t_2^2, \overline{\gamma}^{(2)}(t_2^{-1},t_1 t_2^{-1},t_3 t_2,t_4t_2^2) y) } \mathsf{E}^{\PT}_{\nu}|_{(t_1 t_2^{-1}, t_2^{-1},t_3 t_2,t_4 t_2^2, \overline{\gamma}^{(23)}(t_2^{-1},t_3 t_2,t_4t_2^2) y))} \\
&\cdot \mathsf{V}^{\PT}_{\nu\lambda\varnothing\varnothing}|_{(t_2t_1^{-1},t_1^{-1},t_3 t_1,t_4 t_1^2,\overline{\gamma}^{(3)}(t_2t_1^{-1},t_1^{-1},t_3 t_1,t_4 t_1^2) y)} \\
&= \Big( \sum_n I_{n,0}(L,y) \, q^n \Big) \cdot \Big( \sum_n P_{n,\beta}(L,y) \, q^n \Big). \qedhere
\end{align*}
\end{proof}

\begin{remark}\label{rmk on other insertions}
Let $X$ be a toric Calabi-Yau 4-fold and let $I := I_n(X,\beta)$, $P:=P_n(X,\beta)$. Consider the ``virtual holomorphic Euler characteristic'' of $I$
\begin{align*}
&\chi\Big(I, \widehat{\O}^{\vir}_I \Big) 
:= \chi \Big(I^{T}, \frac{ \O^{\vir}_{I^T} \otimes \sqrt{K_I^{\vir}}^{\frac{1}{2}}|_{I^T} }{\Lambda^{\mdot} \sqrt{N^{\vir}}^{\vee}} \Big) \\
&:=\sum_{Z \in I^{T}} (-1)^{o(\L)|_Z} e\left(\sqrt{\mathrm{Ob}_I|_Z}^f\right)  \frac{\ch\left(\sqrt{K_I^{\vir}|_{Z}}^{\frac{1}{2}}\right)}{\ch\left(\Lambda^{\mdot} \sqrt{N^{\vir}|_{Z}}^{\vee}\right)}  \td\left(\sqrt{T_I^{\vir}|_{Z}}^{f}\right),
\end{align*}
and its stable pairs analogue with $I$ replaced by $P$. Then one can develop a (simpler) vertex formalism for these invariants. We checked in the case of a single leg  of multiplicity one with a single embedded point that the analogue of the DT/PT correspondence (Conjecture \ref{K-conj intro}) fails for all choices of signs. 

Another natural thing to try is to replace $L$ in Definition \ref{Nekgen} by a $T$-equivariant vector bundle of rank 2 or 3 or a $K$-theory class of rank $-1$ (more precisely: the class of $-L$ where $L$ is a $T$-equivariant line bundle on $X$). In none of these cases there exists an analogue of the DT/PT correspondence either. The special feature of the tautological insertion of Definition \ref{Nekgen} is that, after it is absorbed in the vertex as in Section \ref{sec:Kinsert}, the vertex $\widetilde{\mathsf{v}}_\alpha$ has rank \emph{zero} as we will prove in Proposition \ref{coholimitkey} below.
\end{remark}

\section{Limits of $K$-theoretic conjecture} \label{limitssec}

\subsection{Dimensional reduction} 

Let $X$ be a toric Calabi-Yau 4-fold and $\beta \in H_2(X,\Z)$. Let $Z = \{\{Z_\alpha\}_{\alpha \in V(X)}, \{Z_{\alpha\beta}\}_{\alpha\beta \in E(X)}\}$ be an element of either of the fixed loci $$\bigcup_n I_n(X,\beta)^T, \quad \bigcup_n P_n(X,\beta)^T,$$ where we recall Assumption \ref{assumption} from the introduction. 
We will work in one chart $U_\alpha \cong \C^4$.

Suppose the underlying Cohen-Macaulay curve corresponding to $Z_{\alpha}$ lies scheme theoretically inside the hyperplane $\{x_4 = 0\}$. In the stable pairs case, this implies $Z_\alpha$ is scheme theoretically supported inside $\{x_4=0\}$, however in the DT case $Z_\alpha$ may have embedded points ``sticking out'' of $\{x_4=0\}$.
\begin{proposition} \label{dimredprop}
If $Z_\alpha$ lies scheme theoretically in $\{x_4=0\}$, then $\widetilde{\mathsf{v}}_{Z_\alpha} |_{y = t_4} = \mathsf{V}_{Z_\alpha}^{\mathrm{3D}}$, where $\mathsf{V}_{Z_{\alpha}}^{\mathrm{3D}}$ is the (fully equivariant) DT/PT vertex of \cite[Sect.~4.7--4.9]{MNOP} and \cite[Sect.~4.4--4.6]{PT2}. If the underlying Cohen-Macaulay curve of $Z_\alpha$ lies scheme theoretically in $\{x_4=0\}$, but $Z_\alpha$ does not (which can only happen in the DT case), then  $[-\widetilde{\mathsf{v}}_{Z_\alpha}] |_{y = t_4} = 0$.
\end{proposition}
\begin{proof}
When $Z_\alpha$ lies scheme theoretically inside $\{x_4=0\} \subseteq \C^4 =:U_\alpha$, the statement follows at once by comparing \eqref{defvef2} to \cite[Sect.~4.7--4.9]{MNOP}, \cite[Sect.~4.4--4.6]{PT2}. 

Suppose we consider the DT case and the underlying maximal Cohen-Macaulay curve of $Z_\alpha$ is scheme theoretically supported in $\{x_4=0\}$, but $Z_\alpha$ is not scheme theoretically supported in $\{x_4=0\}$. The vertex $\mathsf{v}_{Z_\alpha}$ does not have $T$-fixed part with positive coefficient by Lemma \ref{Ksqrt}. Therefore, it suffices to consider the $T$-fixed terms arising from setting $y=t_4 = (t_1t_2t_3)^{-1}$ in the $y$-dependent part of $\widetilde{\mathsf{v}}_{Z_\alpha}$. 

As usual, we write
$$
Z_\alpha = \sum_{i=1}^{3} \frac{Z_{\alpha\beta_i}}{1-t_i} + W,
$$
where $\beta_1,\beta_2,\beta_3$ are the vertices in $\{x_4=0\}$ neighbouring $\alpha$ and $W$ is a Laurent polynomial in $t_1,t_2,t_3,t_4$. The only terms involving $y$ in $\widetilde{\mathsf{v}}_{Z_\alpha}$ are $-y \overline{W}$. Since $Z_\alpha$ is not scheme theoretically supported inside $\{x_4=0\}$, $W$ contains the term $+t_4$. Setting $y=t_4$ this term contributes 
$$
-y \overline{t_4} = -1.
$$
Furthermore, the underlying maximal Cohen-Macaulay curve of $Z_\alpha$ is contained in $\{x_4 = 0\}$, so all negative terms of $W$ are of the form $-t_1^{w_1} t_2^{w_2}t_3^{w_3}$ with $w_1, w_2, w_3 \geqslant 0$. Therefore $W$ does not contain terms of the form $-t_4$ (which equals $-t_1^{-1} t_2^{-1} t_3^{-1}$). Hence $\widetilde{\mathsf{v}}_{Z_\alpha}|_{y=t_4}$ has negative $T$-fixed part and the proposition follows.
\end{proof}

\begin{remark} \label{dimredrem}
Consider a chart $U_{\alpha\beta} \cong \C^* \times \C^3$ and suppose in both charts $U_\alpha, U_\beta$, the line $L_{\alpha\beta} \cong \PP^1$ is given by $\{x_2 = x_3 = x_4 = 0\}$. Suppose $m_{\alpha\beta}''=0$. Consider a Cohen-Macaulay curve $Z_{\alpha\beta}$ which is scheme theoretically supported on $\{x_4=0\}$. Then, similar to Proposition \ref{dimredprop}, $\widetilde{\mathsf{e}}_{Z_{\alpha\beta}} |_{y = t_4} = \mathsf{E}_{Z_{\alpha\beta}}^{\mathrm{3D}}$, where $\mathsf{E}_{Z_{\alpha\beta}}^{\mathrm{3D}}$ is the fully equivariant DT (and therefore PT) edge  of \cite[Sect.~4.7--4.9]{MNOP}.
\end{remark}

\begin{proof}[Proof of Theorem \ref{dimred intro}]
The first part of Theorem \ref{dimred intro} is an immediate corollary of Proposition \ref{dimredprop}. Note that on RHS we obtain $-q$ due to our choice of signs\,\footnote{Choosing all signs of all $T$-fixed points which are scheme theoretically supported on $\{x_4=0\}$ equal to $+1$ amounts to replacing $-q$ by $q$ on RHS of \eqref{dimredeqns}.}.

 For the second part of Theorem \ref{dimred intro}, we choose signs as in Conjecture \ref{K-conj intro} and we assume this can be done compatibly with the choice of signs of all $T$-fixed points which are scheme theoretically supported on $\{x_4=0\}$ as stated in the theorem. Then the second part of the theorem follows.
\end{proof}

\begin{proof}[Proof of Theorem \ref{dimredcor}]
We recall the vertex formalism for $K$-theoretic DT invariants of toric 3-folds from \cite{NO, O, Arb} (the stable pairs case is similar). We have
\begin{equation*} 
\chi(I_n(D,\beta), \widehat{\O}_I^{\vir}) = \sum_{Z \in I_n(D,\beta)^{(\C^*)^3}} e(\mathrm{Ob}_I^f|_Z) \frac{\ch( (K_I^{\vir}|_Z)^{\frac{1}{2}} )}{\ch(\Lambda^{\mdot} (N^{\vir}|_Z)^{\vee})} \td((T_{I}^{\vir}|_Z)^f),
\end{equation*}
where $T_{I}^{\vir}|_Z$ is the virtual tangent bundle, i.e.~dual perfect obstruction theory, of $I:= I_n(D,\beta)$ at $Z$, $\mathrm{Ob}_I := h^1(T_I^{\vir})$, and the square root exists by \cite[Sect.~6]{NO}. 
Note that for different choices of square roots $(K_I^{\vir}|_Z)^{\frac{1}{2}}$, the first Chern class (modulo torsion) does not change, so the invariants remain the same (see also \cite[Section 2.5]{Arb}). 
Moreover $N^{\vir}|_Z$ denotes the $(\C^*)^3$-moving part of $T_{I}^{\vir}|_Z$ and $(\cdot)^f$ denotes $(\C^*)^3$-fixed part.
By \cite[Lem.~6]{MNOP}, $T_{I}^{\vir}|_Z$ has no $(\C^*)^3$-fixed terms with positive coefficients\,\footnote{In fact, it has no $(\C^*)^3$-fixed terms with negative coefficient either by \cite[Lem.~8]{MNOP}.}, hence
$$
\chi(I_n(D,\beta), \widehat{\O}_I^{\vir}) = \sum_{Z \in I_n(D,\beta)^{(\C^*)^3}} \frac{\ch( (K_I^{\vir}|_Z)^{\frac{1}{2}} )}{\ch(\Lambda^{\mdot} (T_I^{\vir}|_Z)^{\vee})}.
$$
From \eqref{essential} and \eqref{eqnK}, we deduce\footnote{This is the $K$-theoretic vertex formalism for DT theory on toric 3-folds \cite{NO, O, Arb}. See \cite{MNOP} for the (original) cohomological case.} 
\begin{align*}
\chi(I_n(D,\beta), \widehat{\O}_I^{\vir}) = \sum_{Z \in I_n(D,\beta)^{T}} \bigg( \prod_{\alpha \in V(D)} [-\mathsf{V}_{Z_\alpha}^{\mathrm{3D},\DT}] \bigg) \bigg(  \prod_{\alpha\beta \in E(D)}  [-\mathsf{E}_{Z_{\alpha\beta}}^{\mathrm{3D},\DT}] \bigg),
\end{align*}
where the sums are over all $T$-fixed points $Z =\big(\{Z_\alpha\}_{\alpha \in V(D)}, \{Z_{\alpha\beta}\}_{\alpha\beta \in E(D)}\big)$ and  $\mathsf{V}^{\mathrm{3D},\DT}_{Z_\alpha}, \mathsf{E}^{\mathrm{3D},\DT}_{Z_{\alpha\beta}}$ are evaluated at the characters of the $(\C^*)^3$-action on $U_\alpha\cap D$, $U_{\alpha\beta}\cap D$ respectively.

The generating function $\sum_n \chi(I_n(D,\beta), \widehat{\O}_I^{\vir}) \, q^n$ is calculated using the $K$-theoretic 3-fold DT vertex $\mathsf{V}^{\mathrm{3D},\DT}_{\lambda\mu\nu}(t,q)$ much like in the calculation after Definition \ref{vertexdef}. Since the DT/PT edge coincide \cite[Sect.~0.4]{PT2}, the result follows from Theorem \ref{dimred intro} and a calculation similar to the proof of Theorem \ref{globaltoricKDTPT}.
\end{proof}

\subsection{Cohomological limit I}\label{coho limi 1}

Again, let $Z = \{\{Z_\alpha\}_{\alpha \in V(X)}, \{Z_{\alpha\beta}\}_{\alpha\beta \in E(X)}\}$ be an element of either of the fixed loci $$\bigcup_n I_n(X,\beta)^T, \quad \bigcup_n P_n(X,\beta)^T,$$ where we recall Assumption \ref{assumption} from the introduction. 
We will work in one chart $U_\alpha \cong \C^4$ or $U_{\alpha\beta} \cong \C^* \times \C^3$ with standard torus action \eqref{standardaction}.

\begin{proposition} \label{coholimitkey}
For any $\alpha \in V(X)$ and $\alpha\beta \in E(X)$, we have
\begin{align*}
\widetilde{\mathsf{v}}_{Z_\alpha} |_{(1,1,1,1,1)} = 0, \quad \widetilde{\mathsf{e}}_{Z_{\alpha\beta}} |_{(1,1,1,1,1)} = 0,
\end{align*}
i.e.~the ranks of $\widetilde{\mathsf{v}}_{Z_\alpha}$ and $\widetilde{\mathsf{e}}_{Z_{\alpha\beta}}$ are zero.

Let $Z_{\CM,\alpha}$ be the underlying Cohen-Macaulay curve of $Z_\alpha$ and denote by $\lambda, \mu, \nu, \rho$ its asymptotic plane partitions. In the DT case, define 
$$
W_\alpha := \sum_{w \in Z_\alpha \setminus Z_{\CM,\alpha }} t^w +  \sum_{w \in Z_{\CM, \alpha}} \big( 1 - \# \{\mathrm{legs \ containing  \ } w \} \big) \, t^w.
$$
In the stable pairs case, define
$$
W_\alpha := \sum_{w \in B^{(\alpha)}} t^w +  \sum_{w \in Z_{\CM, \alpha}} \big( 1 - \# \{\mathrm{legs \ containing  \ } w \} \big) \, t^w,
$$
where $B^{(\alpha)}$ is the box configuration corresponding to the fixed point $Z_\alpha$ \eqref{PTZalpha}. Then the terms involving $y$ in the Laurent polynomial $\widetilde{\mathsf{v}}_{Z_\alpha}$ are $- y \overline{W}_\alpha$. 

Suppose $\PP^1 \cong L_{\alpha\beta} = \{x_2=x_3=x_4 = 0\}$, i.e.~leg $Z_{\alpha\beta}$ lies along the $x_1$-axis. Then the terms involving $y$ in the Laurent polynomial  $\widetilde{\mathsf{e}}_{Z_{\alpha\beta}}$ are precisely
\begin{equation} \label{ytermsetilde}
-y \Bigg( \overline{Z}_{\alpha\beta} - \frac{\partial}{\partial t_1} \Big|_{t_1=1} \overline{Z_{\alpha\beta} \Big|_{(t_2t_1^{-m_{\alpha\beta}},t_3t_1^{-m'_{\alpha\beta}},t_4 t_1^{-m''_{\alpha\beta}})}} + O(t_1-1) \Bigg).
\end{equation}
\end{proposition}

\begin{proof}
By definition of $W_\alpha$, we have
$$
Z_{\alpha} = \sum_{i=1}^4 \frac{Z_{\alpha\beta_i}(t_{i'},t_{i''},t_{i'''})}{1-t_i} + W_{\alpha},
$$
where $Z_{\alpha \beta_i}$, $W_{\alpha}$ are all Laurent \emph{polynomials}. Plugging into \eqref{defvef2} gives the following identity in $\Z[t_1^{\pm1}, t_2^{\pm1}, t_3^{\pm1}, y^{\pm 1}]$
$$
\widetilde{\mathsf{v}}_{Z_\alpha}  - (W_{\alpha} - y \overline{W}_{\alpha})  \equiv 0 \mod (1-t_1,1-t_2,1-t_3,1-(t_1t_2t_3)^{-1}).
$$
This implies
$$
\widetilde{\mathsf{v}}_{Z_\alpha} |_{(1,1,1,1,1)} =  \big(  W_{\alpha} - y \overline{W}_{\alpha} \big) |_{(1,1,1,1,1)} = 0.
$$
Moreover, we find that the only terms in $\widetilde{\mathsf{v}}_{Z_\alpha}$ containing $y$ after redistribution are $- y \overline{W}_\alpha$. 

Next we turn our attention to the edge term $\widetilde{\mathsf{e}}_{Z_{\alpha\beta}}$ defined in \eqref{defvef2}. Multiply numerator and denominator of  $\widetilde{\mathsf{e}}_{Z_{\alpha\beta}}$ by $t_1$ so the denominator becomes $t_1-1$. Since numerator and denominator both contain a zero at $t_1=1$, the Laurent polynomial $\widetilde{\mathsf{e}}_{Z_{\alpha\beta}}$ is of the following form
\begin{equation} \label{numerator}
-y \overline{Z}_{\alpha\beta} - \frac{\partial}{\partial t_1} \Big|_{t_1=1} \Big( t_1 \widetilde{\mathsf{f}}_{\alpha\beta} \Big|_{(t_1^{-1}, t_2 t_{1}^{-m_{\alpha\beta}}, t_3 t_{1}^{-m_{\alpha\beta}'}, t_4 t_{1}^{-m_{\alpha\beta}''})} \Big) + O(t_1-1),
\end{equation}
where the term $O(t_1-1)$ obviously has rank 0. 
Next, we write 
$$
Z_{\alpha\beta} = \sum_{j,k \geqslant 1} \sum_{l=1}^{\lambda_{\alpha\beta}} t_2^{j-1} t_3^{k-1} t_4^{l-1},
$$
where $\lambda_{\alpha\beta}$ is the finite plane partition describing $Z_{\alpha\beta}$. Substituting into \eqref{numerator}, one easily finds
$$
\widetilde{\mathsf{e}}_{Z_{\alpha\beta}}|_{(1,1,1,1,1)} = 0.
$$
Moreover, the terms involving $y$ in \eqref{numerator} are 
$$
- y \overline{Z}_{\alpha\beta}  + y \frac{\partial}{\partial t_1} \Big|_{t_1=1} \overline{Z_{\alpha\beta} \Big|_{(t_2t_1^{-m_{\alpha\beta}},t_3t_1^{-m'_{\alpha\beta}},t_4 t_1^{-m''_{\alpha\beta}})}}
$$
modulo multiples of $(t_1-1)$. This yields the result.
\end{proof}

We set $t_i = e^{b \lambda_i}$ for all $i=1,2,3,4$ and $y = e^{b m}$. The relation $t_1t_2t_3t_4=1$ translates into $\lambda_1+\lambda_2+\lambda_3+\lambda_4=0$. We are interested in the limit $b \rightarrow 0$.

\begin{proposition} \label{limitexists}
For any $\alpha \in V(X)$ and $\alpha\beta \in E(X)$, the following limits \begin{align*}
\lim_{b \rightarrow 0} [-\widetilde{\mathsf{v}}_{Z_\alpha}] |_{t_i = e^{b \lambda_i}, y=e^{mb} }, \quad \lim_{b \rightarrow 0} [-\widetilde{\mathsf{e}}_{Z_{\alpha\beta}}] |_{t_i = e^{b \lambda_i}, y = e^{mb}}
\end{align*}
exist in $\Q(\lambda_1,\lambda_2,\lambda_3,\lambda_4,m) / (\lambda_1+\lambda_2+\lambda_3+\lambda_4)$.
\end{proposition}
\begin{proof}
Using multi-index notation for $\tau := (t_1,t_2,t_3,t_4,y)$, where $t_1t_2t_3t_4=1$, we write
\begin{align*}
\widetilde{\mathsf{v}}_{Z_\alpha} &= \sum_{v} \tau^{v} - \sum_{w} \tau^{w}.
\end{align*}
These sums are finite by Lemma \ref{Ksqrt}. This representation is unique when we require that the sequences $\{v\}$, $\{w\}$ have no elements in common.
Proposition \ref{coholimitkey} implies that the number of $+$ (i.e.~``deformation'') and $-$ (i.e.~``obstruction'') terms in both expressions are equal, i.e.
$$
\sum_v 1 - \sum_w 1 = 0. 
$$
Recall that $\widetilde{\mathsf{v}}_{Z_\alpha}$ has no $T$-fixed part with positive coefficient (Lemma \ref{Ksqrt}). When one of the $w$ is zero, $[-\widetilde{\mathsf{v}}_{Z_\alpha}] = 0$ and the proposition is clear. Next, write the components of the weight vectors in the sum as follows
$$
v = ( v_1 , v_2, v_3, v_4, v_m )
$$
and similarly for $w$. Then
$$
[- \widetilde{\mathsf{v}}_{Z_\alpha}] |_{t_i = e^{b \lambda_i}, y=e^{mb} } = \frac{\prod_{w}  \big(  (w_1 \lambda_1 + w_2 \lambda_2 +  w_3 \lambda_3 +  w_4 \lambda_4 +  w_m m )\,b + O(b^2) \big)     }{\prod_{v}  \big( (v_1 \lambda_1 + v_2 \lambda_2 + v_3 \lambda_3 +  v_4 \lambda_4 +  v_m m )\,b + O(b^2) \big)  }.
$$
Since the number of factors in numerator and denominator is equal, say $N = \sum_v 1 = \sum_w 1$, we can divide numerator and denominator by $b^N$ and deduce that the limit exists and equals
\begin{equation} \label{lasteq}
\lim_{b \rightarrow 0} [- \widetilde{\mathsf{v}}_{Z_\alpha}] |_{t_i = e^{b \lambda_i}, y=e^{mb} } =  \frac{\prod_{w}  (w_1 \lambda_1 + w_2 \lambda_2 +  w_3 \lambda_3 +  w_4 \lambda_4 +  w_m m )      }{\prod_{v}  (v_1 \lambda_1 + v_2 \lambda_2 + v_3 \lambda_3 +  v_4 \lambda_4 +  v_m m )  }.
\end{equation}
The proof for $\widetilde{\mathsf{e}}_{Z_{\alpha\beta}}$ is similar.
\end{proof}

\begin{proof}[Proof of Theorem \ref{DT/PT tauto}]
Consider the generating series
$$
\sum_{n}  I_{n,\beta}(L,y) \big|_{t_i = e^{b \lambda_i},y = e^{bm}} \, q^n, \quad \sum_{n}  P_{n,\beta}(L,y) \big|_{t_i = e^{b \lambda_i},y = e^{bm}} \, q^n.
$$
Both series are calculated by the vertex formalism of Theorem \ref{vertexthm}. Proposition \ref{limitexists} implies that the limits $b \rightarrow 0$ exist. Recall that for any equivariant line bundle $\mathcal{L}$ we have \eqref{essential}
\begin{equation*} 
\frac{\ch(\Lambda^\mdot \mathcal{L}^*)}{ \ch((\det \mathcal{L}^*)^{\frac{1}{2}})} = \frac{1 - e^{-c_1(\mathcal{L})}}{ e^{-\frac{1}{2} c_1(\mathcal{L}) }} = e^{\frac{c_1(\mathcal{L})}{2}} - e^{-\frac{c_1(\mathcal{L})}{2}}.
\end{equation*}
Let $\tau := (t_1,t_2,t_3,t_4,y)$ and use multi-index notation. If $\mathcal{L} = \tau^w$, where $w = (w_1,w_2,w_3,w_4,w_m)$, we obtain
\begin{align*}
\frac{\ch(\Lambda^\mdot \mathcal{L}^*)}{ \ch((\det \mathcal{L}^*)^{\frac{1}{2}})} \Big|_{t_i = e^{b \lambda_i}, y=e^{bm} } &= (w_1 \lambda_1 + w_2 \lambda_2 + w_3 \lambda_3 + w_4 \lambda_4 + w_m m ) b + O(b^2) \\
&= e(\mathcal{L}) \, b + O(b^2).
\end{align*}
Therefore, in the DT case, we have
\begin{align*}
&\frac{\ch\left(\sqrt{K_I^{\vir}|_{Z}}^{\frac{1}{2}}\right)}{\ch\left(\Lambda^{\mdot} \sqrt{T_I^{\vir}|_{Z}}^{\vee}\right)}   \frac{\ch(\Lambda^{\mdot} (L^{[n]}|_{Z} \otimes y^{-1}))}{\ch((\det(L^{[n]} |_{Z} \otimes y^{-1}))^{\frac{1}{2}})} \Bigg|_{t_i = e^{b \lambda_i}, y=e^{bm} } \\
&= b^{N} \cdot \Big( e\Big( - \sqrt{T_I^{\vir}|_Z}\Big) + O(b) \Big) \cdot  \Big( e(R\Gamma(X,L \otimes \O_Z)^\vee \otimes e^m) + O(b) \Big) \\
&= b^{N} \cdot \Big( \frac{\sqrt{(-1)^{\frac{1}{2}\mathrm{ext}^{2}(I_Z,I_Z)}e\big(\Ext^{2}(I_Z,I_Z)\big)}}{e\big(\Ext^{1}(I_Z,I_Z)\big)} + O(b) \Big) \cdot  \Big( e(R\Gamma(X,L \otimes \O_Z)^\vee \otimes e^m) + O(b) \Big)
\end{align*}
for some $N \in \Z$ (and similarly in the PT case). By equation \eqref{lasteq} in the proof of Proposition \ref{limitexists}, we know $N=0$. Taking $b \to 0$ proves the first part of the theorem.

Next assume Conjecture \ref{K-conj intro} holds. The second part of the theorem follows from Theorem \ref{globaltoricKDTPT}.
\end{proof}

\iffalse
\begin{remark} \label{cohoarrow}
%Under the same hypotheses as in Theorem \ref{dimred intro}, Proposition \ref{dimredprop} implies the cohomological DT/PT correspondence for toric 3-folds in the 2-leg case \cite[Conj.~4]{PT2}
Under the cohomological limit taken in Proposition \ref{limitexists}, Proposition \ref{dimredprop} will give a dimensional reduction of 4-fold cohomological DT/PT vertex to 3-fold cohomological DT/PT vertex. 
By the vertex formalism (as in Theorem \ref{dimredcor}), we conclude the 4-fold cohomological 
DT/PT correspondence as stated in Theorem \ref{DT/PT tauto} implies the 3-fold cohomological 
DT/PT correspondence. This gives the second diagonal arrow of Figure 1 in the introduction.
\end{remark}
\fi

\begin{remark} \label{cohoarrow}
Taking the cohomological limit of Proposition \ref{limitexists} and setting $m = -\lambda_1-\lambda_2-\lambda_3$, one recovers the cohomological 3-fold DT/PT vertex from the cohomological 4-fold DT/PT vertex (this follows from Proposition \ref{dimredprop}). Using the vertex formalism, the 4-fold cohomological DT/PT correspondence therefore implies the 3-fold cohomological DT/PT correspondence. 
This gives the second diagonal arrow of Figure 1 in the introduction.
\end{remark}

\iffalse
\begin{remark} \label{cohoarrow}
%Under the same hypotheses as in Theorem \ref{dimred intro}, Proposition \ref{dimredprop} implies the cohomological DT/PT correspondence for toric 3-folds in the 2-leg case \cite[Conj.~4]{PT2}
In the setting of Lemma \ref{dimred intro}, Proposition \ref{dimredprop} implies the cohomological DT/PT correspondence for toric 3-folds in the 2-leg case \cite[Conj.~4]{PT2}
$$
 \mathsf{V}_{\lambda\mu\varnothing}^{\mathrm{3D},\mathrm{coho},\DT}(q)  =  \mathsf{V}_{\lambda\mu\varnothing}^{\mathrm{3D},\mathrm{coho},\PT}(q)    \, \mathsf{V}_{\varnothing\varnothing\varnothing}^{\mathrm{3D},\mathrm{coho},\DT}(q),
$$
where $\mathsf{V}_{\lambda\mu\nu}^{\mathrm{3D},\mathrm{coho},\DT}(q)$,  $\mathsf{V}_{\lambda\mu\nu}^{\mathrm{3D},\mathrm{coho},\PT}(q)$ are the (fully equivariant) DT/PT 3-fold vertex of \cite[Sect.~4]{MNOP2}, \cite[Sect.~4.7]{PT2}. Under the same hypotheses as in Corollary \ref{dimredcor}, we conclude
$$
\frac{\sum_{n} I^{\mathrm{coho}}_{n,\beta}(D) \, q^n}{\sum_{n} I^{\mathrm{coho}}_{n,0}(D) \, q^n} = \sum_{n} P^{\mathrm{coho}}_{n,\beta}(D) \, q^n,
$$
where $I^{\mathrm{coho}}_{n,\beta}(D)$, $P^{\mathrm{coho}}_{n,\beta}(D)$ are the cohomological DT/PT invariants of $D$ \cite{MNOP, PT2}. This gives the second diagonal arrow of Figure 1 in the introduction.
\end{remark}
\fi

\subsection{Cohomological limit II}\label{coho limi 2}

As before, let $Z =\{\{Z_\alpha\}_{\alpha \in V(X)}, \{Z_{\alpha\beta}\}_{\alpha\beta \in E(X)}\}$ be an element of either of the fixed loci $$\bigcup_n I_n(X,\beta)^T, \quad \bigcup_n P_n(X,\beta)^T,$$ where we recall Assumption \ref{assumption} from the introduction. 
We will work in one chart $U_\alpha \cong \C^4$ or $U_{\alpha\beta} \cong \C^* \times \C^3$ with standard torus action \eqref{standardaction}.
In the DT case, $Z_\alpha$ is a point- or curve-like solid partition, whose renormalized volume we denote by $|Z_\alpha|$. In the stable pairs case, $Z_\alpha$ consists of a Cohen-Macaulay support curve $Z_{\CM, \alpha}$ together with a box configuration $B^{(\alpha)}$ \eqref{PTZalpha}. We denote the sum of the renormalized volume of $Z_{\CM,\alpha}$ and the length of $B^{(\alpha)}$ by $|Z_\alpha|$ as well.

In this section, we set $t_i = e^{b \lambda_i}$ for all $i=1,2,3,4$, $y = e^{b m}$, $Q = q m$, and take the double limit $b \rightarrow 0$, $m \rightarrow \infty$. In \eqref{defcohoinvII}, we recalled the definition of the cohomological DT/PT invariants $I_{n,\beta}^{\mathrm{coho}}, P_{n,\beta}^{\mathrm{coho}}$ studied in \cite{CK2}. In \cite{CK2}, we defined 
$$
\mathsf{V}_{\lambda\mu\nu\rho}^{\mathrm{coho},\DT}, \quad \mathsf{V}_{\lambda\mu\nu\rho}^{\mathrm{coho},\PT}, \quad \mathsf{E}_{\lambda}^{\mathrm{coho},\DT}, \quad \mathsf{E}_{\lambda}^{\mathrm{coho},\PT},
$$
which are defined precisely as in Definition \ref{vertexdef} but with the Nekrasov bracket $[\cdot]$ replaced by $T$-equivariant Euler class $e(\cdot)$.

\begin{proposition} \label{doublelimit}
For any $\alpha \in V(X)$ and $\alpha\beta \in E(X)$, we have
\begin{align*}
\lim_{b \rightarrow 0 \atop m \rightarrow \infty} \big([-\widetilde{\mathsf{v}}_{Z_\alpha}] \ q^{|Z_\alpha|}\big)  \big|_{t_i = e^{b \lambda_i}, y=e^{mb}, Q=m q} = e(-\mathsf{V}_{Z_\alpha}^{\mathrm{coho}}) \, Q^{|Z_\alpha|},  \\
\lim_{b \rightarrow 0 \atop m \rightarrow \infty} \big([-\widetilde{\mathsf{e}}_{Z_{\alpha\beta}}] \ q^{f(\alpha,\beta)}\big)  \big|_{t_i = e^{b \lambda_i}, y=e^{mb}, Q=m q} = e(-\mathsf{E}_{Z_{\alpha\beta}}^{\mathrm{coho}}) \, Q^{f(\alpha,\beta)},
\end{align*}
where $f(\alpha,\beta)$ was defined in \eqref{fab}.
\end{proposition}
\begin{proof}
We continue using the notation of the proof of Propositions \ref{coholimitkey} and \ref{limitexists}. We already showed
\begin{align}\label{intermedlim}
\lim_{b \rightarrow 0} [- \widetilde{\mathsf{v}}_{Z_\alpha}] |_{t_i = e^{b \lambda_i}, y=e^{mb} } \, q^{|W_\alpha|}  = \frac{\prod_{w}  (w_1 \lambda_1 + w_2 \lambda_2 +  w_3 \lambda_3 +  w_4 \lambda_4 +  w_m m ) }{\prod_{v}  (v_1 \lambda_1 + v_2 \lambda_2 +  v_3  \lambda_3 + v_4 \lambda_4 +  v_m m )  } \,  \Big(\frac{Q}{m}\Big)^{|W_\alpha|},
\end{align}
where $W_\alpha$ was defined in the statement of Proposition \ref{coholimitkey}. In fact, $y$ always appears in $\widetilde{\mathsf{v}}_{Z_\alpha}$, $\widetilde{\mathsf{e}}_{Z_{\alpha\beta}}$ with power $+1$, so $w_m$, $v_m$ are elements of $\{0,1\}$.

As before, we set $\tau:=(t_1,t_2,t_3,t_4,y)$ and we write 
$$
W_\alpha =  \sum_{a} \tau^{a} - \sum_{c} \tau^{c},
$$
where the collections of weights $\{a\}$ and $\{c\}$ have no elements in common. Observe that the rank of $W_\alpha$ equals the renormalized volume $|Z_\alpha|$ (Definition \ref{solid}). By Proposition \ref{coholimitkey}, 
the terms of (\ref{intermedlim}) involving $m$ (i.e.~$w_m \neq 0$ or $v_m\neq0$) are precisely
\begin{align*}
&\frac{\prod_{a} (-a_1 \lambda_1 - a_2 \lambda_2 -  a_3 \lambda_3 -  a_4 \lambda_4 + m )   }{\prod_{c} (-c_1 \lambda_1 - c_2 \lambda_2 -  c_3 \lambda_3 -  c_4 \lambda_4  + m) } m^{-\sum_a 1+ \sum_c 1 } \, Q^{|Z_\alpha|} \\  
&=\frac{\prod_a (-a_1 \frac{\lambda_1}{m} - a_2 \frac{\lambda_2}{m} -  a_3 \frac{\lambda_3}{m} - a_4 \frac{\lambda_4}{m} + 1 )   }{\prod_{c} (-c_1 \frac{\lambda_1}{m} - c_2 \frac{\lambda_2}{m} -  c_3 \frac{\lambda_3}{m} -  c_4 \frac{\lambda_4}{m} + 1 ) }  \, Q^{|Z_\alpha|}. 
\end{align*}
Therefore, sending $m \rightarrow \infty$, this term becomes $Q^{|Z_\alpha|}$. 
As we saw in the proof of Theorem \ref{DT/PT tauto}, the terms of (\ref{intermedlim}) which do \emph{not} involving $m$ (i.e.~$w_m = v_m=0$) together are equal to $e(-\mathsf{V}_{Z_\alpha}^{\mathrm{coho}})$. So in total, we have
$$
\lim_{b \rightarrow 0 \atop m \rightarrow \infty} \big([-\widetilde{\mathsf{v}}_{Z_\alpha}] \ q^{|Z_\alpha|}\big)  \big|_{t_i = e^{b \lambda_i}, y=e^{mb}, Q=m q} = e(-\mathsf{V}_{Z_\alpha}^{\mathrm{coho}}) \, Q^{|Z_\alpha|}.
$$
 
Next, we turn our attention to $\widetilde{\mathsf{e}}_{Z_{\alpha\beta}}$. Let $\lambda_{\alpha\beta}$ be the asymptotic plane partition corresponding to $Z_{\alpha\beta}$. Denote the term in between brackets in \eqref{ytermsetilde} by
$$
\sum_{a} \tau^{a} - \sum_{c} \tau^{c},
$$
where $\{a\}$ and $\{c\}$ have no elements in common. By \eqref{ytermsetilde}, the rank of this complex is
$$
\sum_{a} 1 - \sum_{c} 1 = f(\alpha,\beta),
$$
where $f(\alpha,\beta)$ is defined in \eqref{fab}. Consequently, the terms of 
$$
\lim_{b \rightarrow 0} [- \widetilde{\mathsf{e}}_{Z_{\alpha\beta}}] |_{t_i = e^{b \lambda_i}, y=e^{mb} }  
$$
involving $m$ are equal to
\begin{align*}
&\frac{\prod_{a} (a_1 \lambda_1 + a_2 \lambda_2 +  a_3 \lambda_3 +  a_4 \lambda_4 + m )   }{\prod_{c} (c_1 \lambda_1 + c_2 \lambda_2 +  c_3 \lambda_3 +  c_4 \lambda_4  + m) } m^{-\sum_{a}1+ \sum_{c} 1 } Q^{f(\alpha,\beta)}  \\
&=\frac{\prod_{a} (a_1 \frac{\lambda_1}{m} + a_2 \frac{\lambda_2}{m} +  a_3 \frac{\lambda_3}{m} +  a_4 \frac{\lambda_4}{m} + 1 )   }{\prod_{c} (c_1 \frac{\lambda_1}{m} + c_2 \frac{\lambda_2}{m} +  c_3 \frac{\lambda_3}{m} +  c_4 \frac{\lambda_4}{m} + 1 ) }  Q^{f(\alpha,\beta)}. 
\end{align*}
Taking $m \rightarrow \infty$, this reduces to $Q^{f(\alpha,\beta)}$. As in the case of the vertex, we conclude
\begin{equation*}
\lim_{b \rightarrow 0 \atop m \rightarrow \infty} \big([-\widetilde{\mathsf{e}}_{Z_{\alpha\beta}}] \ q^{f(\alpha,\beta)}\big)  \big|_{t_i = e^{b \lambda_i}, y=e^{mb}, Q=m q} = e(-\mathsf{E}_{Z_{\alpha\beta}}^{\mathrm{coho}}) \, Q^{f(\alpha,\beta)}. \qedhere
\end{equation*}
\end{proof}

\begin{proof}[Proof of Theorem \ref{coholimit intro}]
The first part of the theorem follows from Theorem \ref{vertexthm} and Proposition \ref{doublelimit}. Moreover, by Proposition \ref{doublelimit} and Definition \ref{vertexdef} we have
\begin{align*}
\lim_{b \rightarrow 0 \atop m \rightarrow \infty} \mathsf{V}_{\lambda\mu\nu\rho}^{\DT}(t,y,q)  \big|_{t_i = e^{b \lambda_i}, y=e^{mb}, Q=m q}  &= \mathsf{V}_{\lambda\mu\nu\rho}^{\mathrm{coho},\DT}(Q), \\
\lim_{b \rightarrow 0 \atop m \rightarrow \infty} \mathsf{V}_{\lambda\mu\nu\rho}^{\PT}(t,y,q)  \big|_{t_i = e^{b \lambda_i}, y=e^{mb}, Q=m q}  &= \mathsf{V}_{\lambda\mu\nu\rho}^{\mathrm{coho},\PT}(Q),
\end{align*}
where $\lambda,\mu,\nu,\rho$ are finite plane partitions and in the stable pairs case, we assume at most two of them are non-empty. Moreover, the choices of signs for RHS are determined by the choices of signs for LHS. We deduce that Conjecture \ref{K-conj intro} implies Conjecture \ref{conjCK2}.
\end{proof}

\appendix

\section{Hilbert schemes of points} \label{appA}

In this appendix, we consider Nekrasov's Conjecture \ref{Nekconj} in the two cohomological limits discussed in Section \ref{coho limi 1}, \ref{coho limi 2} (see also \cite{Nek, NP}).

Let $X$ be a toric Calabi-Yau 4-fold with $T$-equivariant line bundle $L$. By Theorem \ref{DT/PT tauto}, we have
\begin{align*} 
\lim_{b \rightarrow 0 \atop m \rightarrow 0} I_{n,0}(L, y) |_{t_i = e^{b \lambda_i},y = e^{bm}} =  \lim_{m \rightarrow 0} \int_{[\Hilb^n(X)]^{\vir}_{o(\L)}} c_n((L^{[n]})^{\vee} \otimes e^{m}), 
\end{align*}
where the invariants on the RHS are defined by localization \eqref{defcohoinvI}. Since $L^{[n]}$ is a rank $n$ vector bundle, we have 
$$
 c_n((L^{[n]})^{\vee} \otimes e^{m}) = \sum_{i=0}^n c_i((L^{[n]})^{\vee})\,m^{n-i},
$$
and similarly at any $T$-fixed point. Hence\,\footnote{Note that on a smooth projective Calabi-Yau 4-fold, $[\Hilb^n(X)]^{\vir}_{o(\L)}$ has degree $2n$ and the limit $m \rightarrow 0$ would not be needed.}
\begin{align*} 
\lim_{b \rightarrow 0 \atop m \rightarrow 0} I_{n,0}(L, y) |_{t_i = e^{b \lambda_i},y = e^{bm}} = (-1)^n \int_{[\Hilb^n(X)]^{\vir}_{o(\L)}} c_n(L^{[n]}).
\end{align*}
These invariants were studied in \cite{CK1}, where it is conjectured that there exist choices of signs such that the following equation holds
\begin{equation} \label{conjCK1}
\sum_{n=0}^{\infty} q^n \int_{[\Hilb^n(X)]^{\vir}_{o(\L)}} c_n(L^{[n]}) = M(-q)^{\int_X c_1(L)\,c_3(X)},
\end{equation}
where all Chern classes are $T$-equivariant, $\int_X$ denotes $T$-equivariant push-forward to a point, and $$M(q) := \prod_{n=1}^{\infty} \frac{1}{(1-q^n)^n}$$ denotes MacMahon's generating function for plane partitions.  

As noted before in \cite[Sect.~5.2]{Nek}, the conjectural formula \eqref{conjCK1} is a special case of Conjecture \ref{Nekconj} as can be seen as follows.\,\footnote{Unlike \cite{Nek}, which was motivated by physics, our motivation for \eqref{conjCK1} came from our analogous conjecture on smooth projective Calabi-Yau 4-folds \cite[Conj.~1.2]{CK1}.} For any $n \geqslant 1$, we have
\begin{align*}
&\lim_{b \rightarrow 0} \frac{[t_1^nt_2^n][t_1^nt_3^n][t_2^nt_3^n][y^n]}{[t_1^n][t_2^n][t_3^n][t_4^n][y^{\frac{n}{2}}q^n][y^{\frac{n}{2}}q^{-n}]} \Big|_{t_i = e^{b \lambda_i}, y=e^{mb}}  \\
&= \lim_{b \rightarrow 0} \frac{m(\lambda_1 + \lambda_2)(\lambda_1 + \lambda_3)(\lambda_2 + \lambda_3) (bn)^4 + O((bn)^5)}{((\lambda_1 \lambda_2 \lambda_3 \lambda_4) (bn)^4 + O((bn)^5))  (e^{\frac{bmn}{4}} q^{\frac{n}{2}} - e^{-\frac{bmn}{4}} q^{-\frac{n}{2}}  ) (e^{\frac{bmn}{4}} q^{-\frac{n}{2}} - e^{-\frac{bmn}{4}} q^{\frac{n}{2}}  ) } \\
&= \frac{m(\lambda_1 + \lambda_2)(\lambda_1 + \lambda_3)(\lambda_2 + \lambda_3)}{\lambda_1 \lambda_2 \lambda_3( \lambda_1 + \lambda_2 + \lambda_3)  (q^{\frac{n}{2}} - q^{-\frac{n}{2}}  )^2  }.
\end{align*}
Recall the following identity
$$
\mathrm{Exp}\bigg(\frac{q}{(1-q)^2} \bigg) = \prod_{n=1}^{\infty} \frac{1}{(1-q^n)^{n}}.
$$
Let $L|_{\C^4} \cong \O_{\C^4} \otimes t_{1}^{d_1}t_{2}^{d_2}t_{3}^{d_3}t_{4}^{d_4}$. Taking $m = -(d_1 \lambda_1 + d_2 \lambda_2 + d_3 \lambda_3 + d_4 \lambda_4)$ and using Theorem \ref{DT/PT tauto}, we see that Nekrasov's conjecture implies \eqref{conjCK1} for $X = \C^4$. Since LHS and RHS of \eqref{conjCK1} are ``suitably multiplicative'', \eqref{conjCK1} also follows for any toric Calabi-Yau 4-fold $X$ (see \cite[Prop.~3.20]{CK1} for details).

Finally, we consider the following limit (Theorem \ref{coholimit intro})
\begin{equation*} 
\lim_{b \rightarrow 0 \atop m \rightarrow \infty} \Big( I_{n,0}(\O_X, e^{bm}) \, q^n \Big) \Big|_{t_i = e^{b \lambda_i}, Q=mq} =   Q^n \int_{[\Hilb^n(X)]^{\vir}_{o(\L)}} 1,
\end{equation*}
where the RHS is defined by localization, i.e.~\eqref{defcohoinvII}. For any $n \geqslant 1$, we have
\begin{align*}
&\lim_{b \rightarrow 0 \atop m \rightarrow \infty} \frac{[t_1^nt_2^n][t_1^nt_3^n][t_2^nt_3^n][y^n]}{[t_1^n][t_2^n][t_3^n][t_4^n][y^{\frac{n}{2}}q^n][y^{\frac{n}{2}}q^{-n}]} \Big|_{t_i = e^{b \lambda_i}, y=e^{mb}, Q=m q} \\
&= \lim_{m \rightarrow \infty} \frac{m(\lambda_1 + \lambda_2)(\lambda_1 + \lambda_3)(\lambda_2 + \lambda_3)}{\lambda_1 \lambda_2 \lambda_3( \lambda_1 + \lambda_2 + \lambda_3)  }  \frac{\Big( \frac{Q}{m} \Big)^n }{\Big(1 - \Big(\frac{Q}{m}\Big)^n  \Big)^2} \\
&=  \left\{ \begin{array}{cc}  \frac{(\lambda_1 + \lambda_2)(\lambda_1 + \lambda_3)(\lambda_2 + \lambda_3)}{\lambda_1 \lambda_2 \lambda_3( \lambda_1 + \lambda_2 + \lambda_3)} Q & \mathrm{if \ } n=1 \\ 0 & \mathrm{otherwise}. \end{array} \right.
\end{align*}
Therefore, Nekrasov's Conjecture \ref{Nekconj} implies that there exist choices of signs such that the following identity holds
$$
\sum_{n=0}^{\infty} Q^n \int_{[\Hilb^n(\C^4)]^{\vir}_{o(\L)}} 1 = e^{ \frac{(\lambda_1 + \lambda_2)(\lambda_1 + \lambda_3)(\lambda_2 + \lambda_3)}{\lambda_1 \lambda_2 \lambda_3( \lambda_1 + \lambda_2 + \lambda_3)} Q }.
$$
This formula was also originally conjectured by Nekrasov and discussed in \cite[App.~B]{CK1}. Note that the exponent appearing on RHS equals $-\int_{\C^4} c_3(\C^4)$ (interpreted as a $T$-equivariant integral). Therefore, Conjecture \ref{Nekconj} and the vertex formalism together imply that there exist choices of signs such that the following equation holds
$$
\sum_{n=0}^{\infty} Q^n \int_{[\Hilb^n(X)]^{\vir}_{o(\L)}} 1 = e^{ - Q \int_X c_3(X)}.
$$

\section{Local resolved conifold} \label{app:localrescon}

We start with the following lemma, which recovers \cite[Lemma 5]{PT2} after applying dimensional reduction and cohomological limit I.
\begin{lemma}\label{lem: degree 1 PT}
There exist unique choices of signs such that 
\begin{align*}
\mathsf{V}^{\PT}_{(1), \varnothing, \varnothing, \varnothing}(t,y,q)=\Exp\left(\frac{[yt_1]}{[t_1]}q \right).
\end{align*}
\end{lemma}
\begin{proof}
For every length $n$ of the cokernel, there is only one $T$-fixed point, and the character of the corresponding stable pair is
\begin{align*}
Z_n=\frac{1}{1-t_1}+\sum_{i=1}^{n}t_1^{-i}.
\end{align*}
The corresponding vertex term is easily computed as
\begin{align*}
\widetilde{\mathsf{v}}_{n}&=\sum_{i=1}^{n}t_1^{-i}-y\sum_{i=1}^{n}t_1^{i}
\end{align*}
and we choose $(-1)^n$ for the corresponding sign. Then
\begin{align*}
\mathsf{V}^{\PT}_{(1), \varnothing, \varnothing, \varnothing}(t,y,q)&=\sum_{n\geq 0} q^{n}(-1)^n[-\widetilde{\mathsf{v}}_{n}]\\
&=\sum_{n\geq 0} (y^{-\frac{1}{2}}q)^{n}\prod_{i=1}^n\frac{1-yt_1^i}{1-t_1^i}\\
&=\Exp\left(y^{-\frac{1}{2}}q \frac{1-yt_1}{1-t_1}\right)\\
&=\Exp\left(\frac{[yt_1]}{[t_1]} q\right),
\end{align*}
where in the third line we used   \cite[Ex.~5.1.22]{O}.
\end{proof}

Let $X = D \times \C$, where $D = \mathrm{Tot}_{\mathbb{P}^1}(\O(-1)\oplus \O(-1))$ is the resolved conifold. Consider the generating series
$$
\cZ_X(y,q,Q) := \sum_{n,d} P_{n,d[\PP^1]}(\O, y)\, q^n Q^d. 
$$
Using the vertex formalism, we verified Conjecture \ref{localrescon} in the following cases.
\begin{proposition} \label{verif localrescon}
Conjecture \ref{localrescon} holds for curve classes $\beta = d [\mathbb{P}^1]$ with $d=1,2,3,4$ up to the following orders:
\begin{itemize}
\item $d=1$,
\item $d=2$ modulo $q^6$,
\item $d=3$ modulo $q^6$,
\item $d=4$ modulo $q^7$.
\end{itemize}
Moreover, the choices of signs in these verifications are unique and compatible with the signs in Conjecture \ref{K-conj intro} (and therefore also the signs of Theorem \ref{dimred intro} by Remark \ref{compatdimredsgns}).
%\footnote{More precisely, compatibility means the following. For each of the cases listed in this proposition, there exist choices of signs for the edge terms, such that, combined with our previous signs for the vertex terms (Remark \ref{sign expec}), the formula of the conjecture holds.}
\end{proposition}
\begin{proof}
For degree 1 the conjecture is equivalent to 
\begin{align*}
\cZ_{X,1}(y,q) =\frac{[y]}{[t_4][y^{\frac{1}{2}} q] [y^{\frac{1}{2}} q^{-1}]}.
\end{align*}
The edge term for one leg with multiplicity 1 is
\begin{align*}
\tilde{\mathsf{e}}=t_4-y,
\end{align*}
therefore, by Lemma \ref{lem: degree 1 PT}, we conclude
\begin{align*}
\cZ_{X,1}(y,q)&=\mathsf{V}^{\PT}_{(1), \varnothing, \varnothing, \varnothing}(t,y,q)\cdot \mathsf{V}^{\PT}_{(1), \varnothing, \varnothing, \varnothing}(t,y,q)|_{t_1=t_1^{-1}} \cdot q \cdot (-1) [-\tilde{\mathsf{e}}] \\
&= -\frac{[y]}{[t_4]}q \, \Exp\left(\left( \frac{[yt_1]}{[t_1]}+\frac{[yt_1^{-1}]}{[t_1^{-1}]}\right)q  \right)\\
&= -\frac{[y]}{[t_4]}q \, \Exp\left((y^{\frac{1}{2}}+y^{-\frac{1}{2}})q\right)\\
&=\frac{[y]}{[t_4][y^{\frac{1}{2}} q] [y^{\frac{1}{2}} q^{-1}]}.
\end{align*}
The other cases have been checked by an implementation of the vertex formalism in Mathematica.
\end{proof}

Consider the generating series of $K$-theoretic stable pair invariants of the resolved conifold
$$
\cZ_D(q,Q) := \sum_{n,d} \chi(P_n(Y,d[\mathbb{P}^1]),\widehat{\O}^{\vir}_P) \, q^n Q^d. 
$$
We already saw that the 4-fold PT vertex/edge reduce to the 3-fold PT vertex/edge  for all stable pairs scheme theoretically supported on $D \subseteq X$ after setting $y=t_4$ (Proposition \ref{dimredprop} and Remark \ref{dimredrem}). By the same type of argument as in Proposition \ref{dimredprop}, one can show that all stable pairs \emph{not} scheme theoretically supported in $D \subseteq X$ contribute zero after setting $y=t_4$. Therefore\footnote{Recall the origin of the minus sign from Theorem \ref{dimred intro}.}
$$\cZ_X(y,q,Q) \Big|_{y = t_4} =\cZ_D(-q,Q).$$
Hence Conjecture \ref{localrescon} implies 
$$\cZ_D(-q,Q)= \mathrm{Exp}\Bigg( \frac{-qQ}{(1-q/\kappa)(1-q\kappa)} \Bigg), \quad \kappa := (t_1t_2t_3)^{\frac{1}{2}}.$$
This equality was recently proved by Kononov-Okounkov-Osinenko \cite[Sect.~4]{KOO}, 
which gives good evidence for Conjecture \ref{localrescon} from our perspective. 
Applying the preferred limits discussed by Arbesfeld in \cite[Sect.~4]{Arb}, this formula coincides with an expression obtained using the refined topological vertex by Iqbal-Koz\c{c}az-Vafa \cite[Sect.~5.1]{IKV}. More precisely, setting $\widetilde{q} :=  q\kappa$ and $\widetilde{t} := q\kappa^{-1}$ yields \cite[Eqn.~(67)]{IKV}. The formula also coincides with the generating series of motivic stable pair invariants of the resolved conifold obtained by Morrison-Mozgovoy-Nagao-Szendr\H{o}i in \cite[Prop.~4.5]{MMNS}.

Consider the two cohomological generating series for $X = \mathrm{Tot}_{\PP^1}(\O(-1) \oplus \O(-1) \oplus \O)$ defined by \eqref{defcohoinvI} and \eqref{defcohoinvII}:
\begin{align*}
\cZ_X^{\mathrm{coho}}(m,q,Q) &:= \sum_{n,d} P_{n,d[\PP^1]}^{\mathrm{coho}}(\O, m)\, q^n Q^d,  \\
\cZ_X^{\mathrm{coho}}(P,Q) &:= \sum_{n,d} P_{n,d[\PP^1]}^{\mathrm{coho}}\, P^n Q^d.
\end{align*}
Applying cohomological limits I and II (Theorem \ref{DT/PT tauto} and \ref{coholimit intro}), Conjecture \ref{localrescon} and a calculation similar to the one in Appendix \ref{appA} imply
\begin{align*}
\cZ_X^{\mathrm{coho}}(m,q,Q)  &= \lim_{b \to 0} \cZ_X  \Big|_{t_i = e^{b \lambda_i}, y=e^{mb}}  = \Bigg( \prod_{n=1}^{\infty} (1- Q q^n)^n \Bigg)^{\frac{m}{\lambda_4}}, \\
\cZ_X^{\mathrm{coho}}(P,Q)  &= \lim_{b \to 0 \atop m \to \infty} \cZ_X  \Big|_{t_i = e^{b \lambda_i}, y=e^{mb}, P=m q}
=\exp\left(-\frac{PQ}{\lambda_4}\right).
\end{align*}
A wall-crossing interpretation of the first formula is discussed in \cite{CT4}.
Putting $m = \lambda_4$ in the first expression yields the famous formula for the stable pair invariants (or topological string partition function) of the resolved conifold. The second formula was conjectured, and verified up to the same orders as above in \cite[Conj.~2.22]{CK2}.

\end{document}